\documentclass[reqno]{amsart}
\usepackage{amssymb,latexsym}
\usepackage{amsfonts}
\usepackage{color}

\usepackage[pdftex]{graphicx}
\usepackage{epstopdf}
\usepackage{graphics, graphicx, texdraw}
\usepackage{texdraw}
\usepackage{tcolorbox}

\theoremstyle{plain}
\newtheorem{theorem}{Theorem}
\newtheorem{proposition}{Proposition}
\newtheorem{corollary}{Corollary}
\newtheorem{lemma}{Lemma}
\numberwithin{equation}{section}
\newtheorem{remark}{Remark}

\newcommand{\rr}{\mathbb{R}}

\def\w{\omega}

\newcommand{\E}{{\mathcal{E}}}
\newcommand{\Q}{{\mathcal{Q}}}
\newcommand{\wE}{{\widetilde{\mathcal{E}}}}
\newcommand{\wQ}{{\widetilde{\mathcal{Q}}}}

\newenvironment{example}{\par\smallskip\noindent{\sc Example:}}{\par\smallskip}
\newenvironment{assumption}{\begin{quote} {\sc Assumption:}}{\end{quote}}

\makeatletter
\newcommand{\Spvek}[2][r]{%
  \gdef\@VORNE{1}
  \left(\hskip-\arraycolsep%
    \begin{array}{#1}\vekSp@lten{#2}\end{array}%
  \hskip-\arraycolsep\right)}

\def\vekSp@lten#1{\xvekSp@lten#1;vekL@stLine;}
\def\vekL@stLine{vekL@stLine}
\def\xvekSp@lten#1;{\def\temp{#1}%
  \ifx\temp\vekL@stLine
  \else
    \ifnum\@VORNE=1\gdef\@VORNE{0}
    \else\@arraycr\fi%
    #1%
    \expandafter\xvekSp@lten
  \fi}
\makeatother

\begin{document}

\title[Weak solutions and weak diffeomorphisms]{Weak diffeomorphisms and solutions to conservation laws}
\author[ ]{John Holmes, Barbara Keyfitz, Feride Tiglay}
 
\address{Department of Mathematics and Statistics\\Wake Forest University\\ Winston-Salem, NC 27109}
\email{holmesj@wfu.edu}

\address{Department of Mathematics\\The Ohio State University\\ Columbus, OH 43210}
\email{keyfitz.2@osu.edu}

\address{Department of Mathematics\\The Ohio State University  Newark \\ Newark, OH 43055}
\email{tiglay.1@osu.edu }
\begin{abstract}

Evolution equations which describe the changes in a velocity field over time have been classically studied within the Eulerian or Lagrangian frame of reference. Classically, these  frameworks are equivalent descriptions of the same problem, and the equivalence can be demonstrated by constructing particle paths. For hyperbolic conservation laws, we extend the equivalence between these frameworks to weak solutions for a broad class of problems. Our main contribution in this paper is that we develop a new framework to extend the idea of a particle path to scalar equations and to systems in one dimension which do not explicitly include velocity fields. For systems, we use Riemann invariants as the tool to develop an analog to  particle paths. 
\end{abstract}



\date{\today}

\maketitle 

\begin{center}{\em \today}\end{center}

\section{Introduction}
In this paper we study weak solutions of scalar conservation laws,  compressible gas dynamics equations, and systems of two conservation laws  in one space dimension. In particular, we develop the relation between particle paths and weak solutions for these systems. In systems where one of the components represents velocity, it is common to construct the particle paths from the velocity field. We go in the other direction; we show that for many systems where this had not previously been observed, we can find   particle paths, and obtain a conservative system analogous to the Lagrangian formulation of the original system. We show that this can be accomplished even when the particle paths are not differentiable, and when there is no natural velocity field, and we show that the two systems have equivalent weak solutions. 

For compressible gas dynamics equations, explicit formulation of the
particle paths as  diffeomorphisms leads to a conservation law system.  
Using the weak solutions of this system we construct weak particle paths and recover known admissible weak solutions to the original system. We  do this in one space dimension both for periodic solutions and 
for solutions on the real line, with some limitations on what is actually known
about existence of solutions.  The existence of weak particle paths is not surprising; as these systems describe fluid flow, there is a natural formulation of particle paths.  In this instance, our construction is very close to the well known Eulerian-Lagrangian correspondence for weak solutions \cite{Wag1}. 

 For scalar convex conservation laws, including the inviscid Burgers equation,
we develop a particle path approach using diffeomorphisms that extends to weak solutions.
The key is identifying a function of the 
state variable that corresponds to ``velocity'', so that the equation can be 
visualized as a nonlinear transport equation.
Such a  particle path is not a characteristic curve;
in general it is not obvious how to determine the particle path, since these equations often do not describe flow of a substance.

 We extend the construction from scalar convex conservation laws and compressible gas dynamics to systems of two conservation laws in one dimension whose Riemann invariants have a certain structure. These systems may not naturally include a velocity field or have a quantity which can be interpreted as a velocity, and therefore, have  no natural particle paths. Nevertheless, we   identify a quantity  which serves as velocity, and show that so-called particle paths may be constructed. In all three of these case studies, we construct particle paths for weak solutions; these particle paths are not characteristic curves or classical diffeomorphisms.

 For classical solutions, the approach we are using has been studied   for  systems of incompressible fluid flow.  Ebin and Marsden \cite{EM1970} showed that the particle paths, $\gamma$, defined by $\gamma_t = u \circ \gamma$ where $u$ is the velocity field,  are curves in the space of diffeomorphisms. The infinite-dimensional Lie group $G$ of these diffeomorphisms is the ``configuration space'' of the underlying physical system.  Moreover the ``kinetic energy'' of the physical system is used to define an  inner product 
$\langle \cdot, \cdot \rangle$ on the associated Lie algebra $\mathfrak{g}$.
Using right translations this inner product induces a right invariant Riemannian metric on $G$. The motions of the system can be then studied through the geodesic equation defined by the metric on the group of diffeomorphisms $G$. The equation that one obtains by this procedure on $\mathfrak{g}$ is called an \emph{Euler-Arnold equation}. 
However, for weak solutions, rather than considering particle paths as curves through the group of smooth diffeomorphisms, we consider particle  paths as curves through the space of absolutely continuous and invertible isomorphisms. We call these isomorphisms ``weak diffeomorphisms".  

A weak interpretation of the Lagrangian formulation is used in \cite{L2007} for the Hunter-Saxton equation 
\begin{equation}\label{eq:HS}
u_{txx} = -2 u_x u_{xx} - uu_{xxx}
\end{equation}
whose  smooth solutions break down in finite time \cite{HS1991}. This is a geodesic equation  \cite{KM2003} and Lenells  shows in  \cite{L2007}   that  the 
geodesic curves for \eqref{eq:HS}, given by 
$
\varphi_{tt} = \Gamma (\varphi, \varphi_t,\varphi_t), 
$
where $\Gamma$ is a smooth Christoffel map, have closed form solutions
which can be extended past the time at which $\varphi $ ceases to be a diffeomorphism. Moreover he shows that this extension indeed corresponds to weak solutions \eqref{eq:HS}. Though this result is very interesting, it is not immediately obvious how one can generalize to other equations since most equations do not have closed form solutions describing particle  paths.

Another approach that implements the Lagrangian framework to study weak solutions of nonlinear PDEs appears in Bressan-Constantin \cite{BC2005} and Holden-Raynaud \cite{HR2007}. Their approach  transforms  PDEs into  semilinear systems by introducing  new sets of independent and dependent variables. These new variables resolve all singularities that form due to possible wave breaking. The solutions of the new system are obtained as fixed points of a contractive transformation. One remarkable aspect of these works is the construction of bijective maps between Eulerian and Lagrangian formulations. Returning to the original variables, the authors obtain a semigroup of global solutions.

For conservation laws, we find that the equations for particle paths, unlike those of geodesic curves, are not always defined by velocity fields. Nor are these equations given by abstract ODEs, as is the case with particle paths for \eqref{eq:HS}. Rather, it is typical that the equations for the particle paths for solutions of conservation laws are  themselves described by systems of conservation laws. 
In this work, we extend the notion of particle paths to equations without natural particle paths, and find a general framework in which Lagrangian coordinates are natural. We will explore this framework in future work.

The paper is organized as follows. In Sections \ref{sectwo} and \ref{secthree} we study the compressible gas dynamics equations. We consider  the isentropic system in Section \ref{sectwo}. For  the full model considered in Section \ref{secthree}, we study  systems   where either  energy  or  entropy is conserved. The weak formulations for particle paths differ for these two systems, even though they have the same  velocity variable.  In Section \ref{secfour} we  find particle paths, which are not characteristic curves, for weak solutions of  scalar convex  conservation laws. In Section \ref{secfive} we develop a particle path formulation for a class of systems of two conservation laws.

\section{Isentropic Fluid Flow}
\label{sectwo}
Our point of departure is compressible gas dynamics 
in one space dimension. 
Both the isentropic and full (adiabatic or polytropic) models lead to formulations for weak particle paths, and allow selection of
appropriate admissibility criteria.

The Eulerian or ``spatial" formulation for the dynamics of isentropic compressible
gas flow in one space dimension is
\begin{equation} \label{continuity}\begin{split}
\rho_t+(\rho u)_x&=0\,,\\
(\rho u)_t+(\rho u^2+p(\rho))_x&=0\,,
\end{split}\end{equation}
where $\rho=\rho(x,t)$ is the density of the gas, $u(x,t)$ the gas velocity at
a point $(x,t)$ in space-time, and $p=p(\rho)$ the pressure. 
We consider the   Cauchy problem for this system with initial data 
$u(x,0) = u_0(x)$ and $\rho(x,0) = \rho_0(x)$, for $x\in \Omega $ where $\Omega$ is either $\mathbb R$ or $\mathbb T$.  
These equations have been well-studied, classically and recently. The classic reference is Courant and Friedrichs \cite{CF}. 
For the well-posedness theory for systems of conservation laws in a single space variable, modern references are Bressan \cite{Bre:book} and Dafermos \cite{Dafbook}. Appendix A gives more details.

In the Lagrangian or ``referential" 
coordinate system, the isentropic compressible gas dynamics equations become
\begin{equation} \label{cont1}
\begin{split} 
\tau_t -v_y&=0\,,\\
v_t+p(1/\tau)_y&=0\,,
\end{split}\end{equation}
where now $\tau = 1/\rho$ is the specific volume of the fluid and $v$ the velocity, measured
in a coordinate system moving with the fluid.
Specifically, one defines $x'(t)=u(x(t),t)$ as a ``particle path" 
and then $y=\int_{x(t)}^x \rho(s,t)\,dt$.
 The equivalence of \eqref{continuity} and \eqref{cont1} can be extended
to weak solutions, even for flows containing a vacuum, as shown by
Wagner \cite{Wag2}; 
that is, there is a 1-1 correspondence between the {\em admissible} weak solutions of the two systems.
Wagner proved the equivalence  of  the Eulerian and Lagrangian formulations  for weak solutions by constructing weak diffeomorphisms corresponding to particle 
paths. 

We pause here to explain the terminology we are using to describe  admissible weak solutions. 
 Well-posedness for weak solutions of conservation laws requires
additional constraints, usually called {\em admissibility\/} or {\em entropy\/}
conditions.
One such condition for a system of conservation laws,
$u_t+f(u)_x=0$, which we shall use here, is that there exist a smooth convex function,
$\E(u)$, usually called an ``entropy'', and a second function, $\Q(u)$, called an
entropy flux, such that $\E_t+\Q_x=0$ for all smooth solutions of the system.
Then a weak solution is called {\em admissible\/} if
$$ \iint \varphi_t\E(u) +\varphi_x\Q(u) \geq 0\,,$$
for all non-negative test functions $\varphi$.
(Strictly speaking, this inequality defines $\E$-admissibility: admissibility with
respect to a particular admissibility function.)
For the equations of isentropic compressible flow, the specific energy 
(from the third equation in \eqref{gas2}
defined in Section \ref{secthree}), is an example of
an entropy function, with the energy flux serving as the entropy flux.
For the full gas dynamics system there is a
thermodynamic concept of entropy, 
and with a change of sign it is an entropy in the mathematical sense.

To avoid confusion, in this paper we avoid the term ``entropy'' altogether and, 
using terminology introduced by Friedrichs and Lax, \cite{FrLa}, we refer
to $\E,\Q$ as a {\em convex extension\/} 
 if $\E$ is convex in the conserved quantities and
$\E_t+\Q_x=0$ for all smooth solutions of the system at hand.
 
 In our approach, following the lines of research cited above with reference to the Hunter-Saxton equation \eqref{eq:HS}
  the function $\gamma(x,t)$ represents the position of a particle  
that starts at the point $x$ at time $t=0$ and is carried by  the fluid.  
For each $t$, the map $x\mapsto \gamma(x,t)$ is a diffeomorphism for classical solutions
and is a bilipschitz invertible mapping, which we shall refer to as a ``weak diffeomorphism'', for
 weak solutions.
At time $t$,   $u(\gamma(x,t),t)$ is the particle velocity,
and therefore
\begin{equation} \label{gammavel1}
\gamma_t=u(\gamma(x,t),t)\equiv u\circ\gamma\,.
\end{equation}
We obtain an equation for $\gamma$ 
by assuming smooth solutions and writing
 \eqref{continuity} on particle paths. 
 Defining $\zeta = \rho\circ\gamma$, the first equation in
\eqref{continuity} becomes
\begin{equation}\label{cont2}
\zeta_t
=(\rho_t+u\rho_x)\circ\gamma=-(u_x\rho)\circ\gamma\,.
\end{equation}
Differentiating \eqref{gammavel1} with respect to $x$
gives $ \gamma_{tx}=u_x\circ\gamma\cdot\gamma_x$,
so \eqref{cont2} becomes
\begin{equation} \label{zeta} 
\zeta_t=-\frac{\gamma_{tx}}{\gamma_x}\zeta\,.\end{equation}
Note  that $\gamma(x,0)=x$ and $\displaystyle \zeta=\rho \circ \gamma>0$. 
Assuming   $\gamma$ is an
orientation-preserving diffeomorphism so $\gamma_x>0$,  
we can integrate \eqref{zeta} with respect to $t$   to obtain
$
\zeta=\zeta_0/\gamma_x$, and
\begin{equation} \label{rhodef1}
\rho=\zeta\circ\gamma^{-1}=\frac{\rho_0}{\gamma_x}\circ\gamma^{-1}\,.
\end{equation}
We obtain the desired equation for $\gamma$ by writing the second equation of
 \eqref{continuity} along particle paths.
For smooth solutions  \eqref{continuity} is equivalent to
\begin{align} \label{velocity}
u_t+uu_x+\frac{p(\rho)_x}{\rho}&=0\,,
\end{align}
and after a brief calculation we find
\begin{equation}\label{nlwe1}
\gamma_{tt}+\frac{1}{\rho_0}\partial_x\left(p\left(\frac{\rho_0}{\gamma_x}\right)\right)=0\,;
\end{equation}
this is a nonlinear wave equation for $\gamma$ with initial data
\begin{equation} \label{nlwedata1}
\gamma(x,0)=x\,,\quad \gamma_t(x,0)=u_0(x)\,,\quad\textrm{on}  \quad  \Omega \,.
\end{equation}
Because the initial density distribution, $\rho_0$, already figures in \eqref{nlwe1}, 
this problem incorporates the same data as the Cauchy problem for \eqref{continuity}.
Also, since $\rho_0$ does not depend on $t$ we can write equation 
 \eqref{nlwe1}  in divergence form as
\begin{equation}\label{divform1}
\partial_t(\rho_0\gamma_t) + \partial_x\left(p\left(\frac{\rho_0}{\gamma_x}\right)\right)=0\,.
\end{equation}

 
 \subsection{Weak formulation for particle paths}
We write \eqref{divform1} as a first-order system of conservation laws by defining new variables
$\eta=\gamma_x$, $w=\rho_0\gamma_t$ and $v = \rho_0$.
Then the system is
\begin{equation}\label{e-w-v}
\begin{split}
\eta_t-\left(\frac{w}{v}\right)_x&=0\,,\\ 
w_t+\left(p\left(\frac{v}{\eta}\right)\right)_x&=0\,, \\
v_t & = 0\,,
\end{split}
\end{equation}
with the first equation coming from equality of mixed partial derivatives of $\gamma$. 
The initial data are 
\begin{equation} \label{ewvdata}
\eta(x,0) = 1\,,\quad w(x,0) = \rho_0(x) u_0(x)\,,\quad
\textrm{and}\quad v(x,0 ) = \rho_0(x)\,,\quad x\in\Omega\,.\end{equation} 
In the theory of weak solutions to conservation laws the usual requirement is that the functions 
$\gamma_t$ and $\gamma_x$ be defined almost
everywhere, be integrable and be continuous in time as mappings into $L^1_{loc}$.
We note that this is a reasonable expectation for particle paths. 
 
 In this section we prove that there is a one-to-one correspondence between convex extensions
for the isentropic gas dynamics system and the system \eqref{e-w-v} from
which the diffeomorphisms are constructed, and that  admissible solutions of one
system correspond to admissible solutions of the other.
Since the convexity requirement 
is convexity in the conserved variables, $\rho$
and $m=\rho u$ in  \eqref{continuity}, we calculate $\E$ and $\Q$ in $(\rho,m)$ coordinates
in the physical variables.

\begin{proposition} \label{admrhom}
Any convex extension for the equations of
isentropic gas dynamics satisfies the equation
\begin{align} \label{EQ1}
&\E_{\rho\rho}+\frac{2m}{\rho}\E_{\rho m}+
\left(\frac{m^2}{\rho^2}-p'(\rho)\right)\E_{mm}=0\,.
\\
&\Q=\frac{2m}{\rho}\mathcal{E}+\int \left(\E_\rho-\frac{2}{\rho}\E\right) \, dm\,.\notag
\end{align}
Any solution of \eqref{EQ1} that is
convex in the conserved quantities ($\rho$ and $m=\rho u$ in this case) is a
convex extension.
\end{proposition}
The proof is a straightforward calculation, which we omit. 
Convex extensions, given by solutions of \eqref{EQ1}, exist.
One example is the energy function 
$\E=m^2/2\rho + F(\rho)$ with $F$ defined by $F''=p'/\rho$.
Further straightforward calculations prove the following result for the
system \eqref{e-w-v}.
\begin{proposition} \label{admewv}
The system \eqref{e-w-v} possesses convex extensions $\wE$, $\wQ$, functions
of $(\eta,w,v)$; and $\wE$ is of the form
$$ \wE = \eta X\left(\frac{v}{\eta}, \frac{w}{\eta}\right)\,$$
with $X=X(x,y)$ a solution of
\begin{equation} \label{pded} 
X_{xx}+\frac{2y}{x}X_{xy}+\left(\left(\frac{y}{x}\right)^2-p'(x)\right)X_{yy}=0\,.
\end{equation}
Furthermore, the Hessian matrix of 
$\wE$ is positive semidefinite in $(\eta,w,v)$ precisely when
$X$ is convex in its arguments.
\end{proposition}
We have expressed  $X$  as a function of variables that correspond
to density and momentum because in these coordinates the relationship between  \eqref{EQ1}  and \eqref{pded} becomes clear. 
Furthermore, we assert that convex extensions for the isentropic gas dynamics system
correspond exactly to those for the system \eqref{e-w-v}.
The following general result is useful for comparing the two systems. 

\begin{lemma} \label{cofv}
Given an expression $A(\rho,u)_t+B(\rho,u)_x$ in physical variables,
then the correspondence between weak formulations in the two sets of variables
is
\begin{equation} \label{lem1} \iint \psi_t A +\psi_x B\,dx\,dt 
=\iint \widetilde\psi_t\big(\eta A) +\widetilde\psi_y\left[
\left(\frac{\eta w}{v}\right)B-\left(\frac{w}{v}\right) A\right]\,dy\,dt\,, \end{equation}
where $\widetilde\psi(y,t)\equiv\psi\big(\gamma(y,t),t\big)$ and in the expression
on the right, $A$ and $B$ are evaluated at $(v/\eta, w/v)$.
\end{lemma}
\begin{corollary}
The form $\mathcal A_t+\mathcal B_x$ in the
variables $(\eta,w,v)$
corresponds to the weak form of $A_t+B_x$ in physical variables,
where
$$ A =\left( \frac{1}{\eta}\mathcal A\right)\circ \gamma^{-1}\,,
\quad B=\left(\frac{v}{w}\mathcal B-\frac{1}{\eta}\mathcal A\right)\circ\gamma^{-1}\,.$$
\end{corollary}
\noindent 
We can define
\begin{equation} \label{esetup}
\E(\rho,m)= X\left(\rho,m\right) \end{equation} 
where $X$ is
a convex solution of \eqref{pded}, and $\E$ is a possible convex extension
for the physical system \eqref{continuity},
since \eqref{pded} is the same equation as \eqref{EQ1}.
In particular, Proposition \ref{admewv} shows that convex extensions for 
\eqref{e-w-v} 
depend on a combination of the 
diffeomorphism variables in a way
that returns a function of the state variables when translated back to
physical space.

The main theorem of this section is that from an admissible weak solution of 
\eqref{e-w-v} - \eqref{ewvdata} we can 
construct the admissible weak solution of the corresponding 
Cauchy problem for gas dynamics.

The system \eqref{e-w-v} is well-defined and strictly hyperbolic as long as $v>0$
and $\eta>0$.
 The characteristic speeds are $\pm\sqrt{p'(v/\eta)}/\eta$ and $0$;
the first two are genuinely nonlinear under the usual assumptions about $p$,
and the third is linearly degenerate.
Under these hypotheses, known well-posedness theory gives global in time existence and
uniqueness of an admissible weak solution to the Cauchy problem  provided $(\eta_0,w_0,v_0)$ is sufficiently close to
a constant in the total variation norm.  See Appendix A for references. 
We make use of these results in the main theorem of this section.

 \begin{theorem} \label{main2}
 Let $(\eta, w, v)\in C([0, \infty); BV^3 )$ be the admissible weak solution to \eqref{e-w-v}
 and \eqref{ewvdata} with $\eta,v >0$.
The distributional solution $\gamma$ to 
 \begin{equation}\label{gamma10}
\gamma_x(x,t) = \eta(x,t), \quad \gamma _t(x,t) = \frac{w(x,t)}{\rho_0(x)}
\end{equation} is well-defined, absolutely continuous, 
and invertible, and its inverse is absolutely continuous.
Define 
 \begin{equation}\label{rho-and-u}
 \rho := \frac{\rho_0(\gamma^{-1}(x,t),t)}{\eta(\gamma^{-1}(x,t),t)} \ \mbox{ and } \ u :=  
 \frac{w(\gamma^{-1} (x,t),t) }{\rho_0(\gamma^{-1} (x,t)) }. 
 \end{equation}
 Then $(\rho,u)$ is an admissible solution to the 
 Cauchy problem for the isentropic gas dynamics system \eqref{continuity}.
\end{theorem}
\begin{proof}

The function $\gamma$ is well-defined as a consequence of the first equation of
\eqref{e-w-v}; $\gamma$ is absolutely continuous in both $x$ and $t$ as
a consequence of the fact that $\gamma_x$ and $\gamma_t$ are BV
functions.
Furthermore, $\gamma_x$ is positive and bounded above and below, and
hence its inverse is also absolutely continuous.
Thus $\rho$ and $u$ are well defined by \eqref{rho-and-u}.
We begin by considering the weak formulation of \eqref{continuity}:
\begin{align*} 
&I_1 = \iint \psi _t \rho +\psi_x\rho u \,  dxdt\\
&I_2= \iint \phi _t\rho u+\phi_x(\rho u^2+p(\rho)) \,  dxdt.
\end{align*}
Substituting the definitions \eqref{rho-and-u} of  $\rho$ and $u$ gives
\begin{align*} 
&I_1 = \iint\psi _t \left(\frac{v}{\eta}\right) \circ \gamma^{-1} +\psi_x \left(   \frac{w  }{\eta  }
\right) \circ \gamma^{-1} \,  dx dt \\
&I_2=  \iint \phi _t \left(  \frac{w  }{\eta } \right)  \circ \gamma^{-1}  +\phi_x\left ( \frac{w^2  }{v  
\eta }    +p\left(\frac{v}{\eta} \right) \right)\circ \gamma^{-1}\, dxdt.
\end{align*}
We apply the Radon-Nikod\'ym Theorem and 
define   $\widetilde \psi = \psi\circ \gamma$  so that, as measures,  $\psi_t \circ \gamma = \widetilde 
\psi_t - \psi_x  \circ \gamma \gamma_t $, and
$\widetilde \psi_x =\psi_x \circ \gamma  \gamma_x$. This leads to
\begin{align} 
&I_1 =  \iint \widetilde \psi _t    v      \,  dx dt \\
&I_2=  \iint \widetilde \phi _t \left(  \frac{w  }{\eta } \right)  \eta - \widetilde \phi_x  \gamma_t 
\left(  \frac{w  }{\eta } \right)   +\widetilde \phi_x\left ( \frac{w^2  }{v  \eta }    +p\left(\frac{v}
{\eta} \right) \right) \, dxdt = \iint\widetilde \phi _t w   +\widetilde \phi_x      p\left(\frac{\rho_0}
{\eta} \right)   \, dxdt. 
\end{align}
Since  $(\eta, w, v)$  is a weak solution to \eqref{e-w-v}, we find 
\begin{align*} 
&I_1 = -\int  \widetilde \psi (x,0) \rho_0 (x)     \,  dx  \\
&I_2 = -\int \widetilde \phi (x,0) \rho_0(x) u_0(x) \, dx.
\end{align*}
Hence, $\rho$ and $u$ as defined in \eqref{rho-and-u} form a weak solution to the system 
\eqref{continuity} as claimed. 

The one-to-one correspondence between the admissible weak solutions of the isentropic 
equations \eqref{continuity}, and those of  the system \eqref{e-w-v}
is a consequence of the correspondence,  established in \eqref{esetup}, between
convex extensions in the two systems.
Let 
$(\E,\Q)$ be a convex extension for \eqref{continuity},
as defined in Proposition \ref{admrhom}.
Then we wish to establish
$$ I\equiv \iint\phi_t\E +  \phi_x \Q\, dx\,dt\geq 0\,,$$
for any non-negative test function $\phi$.
As before in discussing admissibility, we consider $\E$ and $\Q$ as functions of
$\rho$ and $m=\rho u$; note that
$ m= (w/\eta)\circ\gamma^{-1}$.
Substituting the definitions \eqref{rho-and-u}, we have
$$ I= \iint\phi_t\E\left(\frac{v}{\eta},\frac{w}{\eta}
\right)\circ\gamma^{-1}
 +  \phi_x \Q\left(\frac{v}{\eta},\frac{w}{\eta}
\right)\circ\gamma^{-1}\, dx\,dt\,.
$$
When we change coordinates using the Radon-Nikod\' ym Theorem as before,
and apply Lemma \ref{cofv}, we have
$$ I= \iint\phi_t\eta\E\left(\frac{v}{\eta},\frac{w}{\eta}
\right)
 +  \phi_x \left(\frac{\eta w}{v}\Q\left(\frac{v}{\eta},\frac{w}{\eta}\right)-\frac{w}{v}\E
\right)\, dx\,dt = \iint \phi_t\wE+\phi_x\wQ\,dx\,dt\geq 0\,,
$$
where, from Proposition
\ref{admewv}, $\wE = \eta\E=\eta X$ with corresponding $\wQ$ is clearly a convex extension.
 The final inequality holds if $(\eta,w,v)$ is an admissible solution of
\eqref{e-w-v}.
This completes the proof of Theorem \ref{main2}.
\end{proof}
\section{Compressible gas dynamics equations}
\label{secthree}
The full (or adiabatic) system of equations describing the fluid flow of compressible
gas dynamics in one space dimension is
\begin{equation} \label{gas2}
\begin{split}
\rho_t+(\rho u)_x&=0\,,\\
(\rho u)_t+(\rho u^2+p)_x&=0\,, \\
(\rho E)_t+ \left(\rho  u E  + p u \right)_x&=0\,,
\end{split}
\end{equation}
with ${\displaystyle E=
e+\frac{1}{2}u^2}$ the total energy
and ${\displaystyle  e=\frac{p}{(\alpha -1)\rho}}$ the
internal energy,
expressed in terms of density $\rho=\rho(x,t)$, pressure $p=p(x,t)$ and velocity $u(x,t)$. 
We denote the ratio of specific heats by $\alpha$ rather than the more common
symbol $\gamma$ (typically $1<\alpha<3$). 

A second useful description of compressible gas dynamics can be written in terms of entropy $S=\log p-\alpha \log \rho$ rather than energy and takes the form
\begin{equation} \label{gas}
\begin{split}
\rho_t+(\rho u)_x&=0\,,\\
(\rho u)_t+(\rho u^2+p)_x&=0\,, \\
(\rho S )_t+(\rho u S)_x&=0\,.
\end{split}
\end{equation}
The two systems are equivalent for classical solutions but their weak solutions are different.

We consider   
solutions of the Cauchy problem corresponding to the system \eqref{gas2} or system \eqref{gas} with initial data
\begin{equation}\label{eq:initial}
\rho(x,0)=\rho_0(x), \quad u(x,0)=u_0(x) \ \mbox{ and } \ p(x,0)=p_0(x).
\end{equation}
Next we determine convex extensions for  \eqref{gas2} and \eqref{gas}.

\begin{proposition}\label{entropy-condition}
Every smooth solution of \eqref{gas2} or \eqref{gas} satisfies additional equations of the form $\mathcal E_t + \mathcal Q_x = 0 $ where
\begin{equation}
\mathcal E = \rho X(p\rho^{-\alpha})\,, \quad \mathcal Q = u \rho X(p \rho^{-\alpha})\,, 
\end{equation}
and $X$ is any differentiable function of a single variable. 
For \eqref{gas2} these are all the additional conservation laws.  
For \eqref{gas} these are the only additional conservation laws satisfied by smooth solutions 
except for the specific energy function $\rho E$, and its corresponding flux. 
Moreover if $X'<0<X''$ then $\mathcal{E}$ is convex in the conserved 
variables representing density, momentum and specific energy.
\end{proposition} 
The proof of this is again a straightforward calculation and included in Appendix \ref{B}. 
For classical solutions both systems have the same particle paths.

As in the isentropic case, we
define the function $\gamma(x,t)$ by \eqref{gammavel1}
to represent the position of a particle  
that starts at the point $x$ at time $t=0$ and is carried by  the fluid.  
As in the isentropic fluid equations we integrate the first equation in either system to find 
\begin{equation} \label{rhodef}
\rho=\frac{\rho_0}{\gamma_x}\circ\gamma^{-1}\,.
\end{equation} 
Similarly, integrating  the third equation  along particle paths in either system gives
\begin{equation}\label{hdef}
p=\frac{p_0}{\gamma_x^{\alpha}}\circ\gamma^{-1}.
\end{equation} 
The next step is to obtain a PDE for $\gamma$ by using  the second equation in either system. We differentiate $\gamma_t=u\circ\gamma$ in $t$ and find
$  \gamma_{tt}=(u_t+uu_x)\circ \gamma$.
Then the second equation written in terms of $\gamma$ is
\[ \gamma_{tt}+\left( \frac{1}{\rho}p _x\right)\circ\gamma=0.
\]
Now we use the two expressions  \eqref{rhodef} and \eqref{hdef} that we obtained earlier for $\rho$ and $p$ respectively and find
\begin{equation}\label{nlwe}
\gamma_{tt}+\frac{1}{\rho_0}\left( \frac{p_0}{\gamma_x^{\alpha}}\right)_x=0.
\end{equation}
Moreover, since $\rho_0$ does not depend on $t$ we can write equation  \eqref{nlwe}  in divergence form as
\begin{equation}\label{divform}
\partial_t(\rho_0\gamma_t) + \partial_x\left(\frac{p_0}{\gamma_x^{\alpha}}\right)=0\,.
\end{equation}
To apply conservation law theory, we write  \eqref{divform} as a first order system 
to define  weak solutions. 
Unlike in the isentropic case the equation \eqref{divform} does not appear to provide a natural formulation for weak solutions, see Remark \ref{remark} at the end of the section.  
To obtain a weak formulation for $\gamma$, we consider the two systems in turn.

\subsection{Energy  conserving solutions}
Here we develop the framework for writing  \eqref{gas2} along particle paths using weak diffeomorphisms and we show that this formulation captures admissible weak solutions of the system.
 
We let  $\eta=\gamma_x$ , $w=\rho_0\gamma_t$  and solve for $p_0$ in terms of energy, assumed to be conserved; more precisely, from equation \eqref{hdef} we have 
$p_0 = (p \circ \gamma) \gamma_x^\alpha$. Substituting $p = (\alpha-1) \rho (E - \frac12 u^2) $ gives 
$$
p_0 = (\alpha-1) ( \rho E  - \frac12 \rho u^2) \circ \gamma \gamma_x^\alpha \,.
$$
Using $\gamma _t = u \circ \gamma$, and the definitions of $\eta$ and $ w $, we have  $ u \circ \gamma = \frac{w}{\rho_0} $ and   $ \rho \circ \gamma = \frac{\rho_0}{\eta}$ and therefore 
$$
p_0 = (\alpha-1) \left( \frac{s}{\eta} - \frac12 \frac{w^2}{\eta \rho_0} \right) \eta ^\alpha \,,
$$
where we introduced $s=  \rho_0E \circ \gamma$.   Thus we arrive at the system 
\begin{equation}\label{sys2}
\begin{split}
\eta_t-\left(\frac{w}{ v }\right)_x&=0\,,\\ 
w_t+ (\alpha-1) \left(  {\frac{s}\eta - \frac12  \frac{w^2}{ v \eta} }{  }\right)_x&=0\,, \\ 
s_t + (\alpha-1) \left( \frac{sw}{\eta  v }   - \frac12 \frac{w^3}{\eta  v ^2}  \right)_x & = 0\,,
\\
v_t & = 0\,,
\end{split}
\end{equation}
 where the third equation results from imposing conservation of energy and the fourth identifies $v$ with $\rho_0$. This system, paired with the initial data 
 \begin{equation}\label{eq:ini2}
 \eta(x,0) = 1\,,\;\;  w(x,0) = \rho_0u_0\,,\;\; s(x,0)=\frac{1}{2}\rho_0 u_0^2+ \frac{p_0}{\alpha -1}\,,  \mbox{ and }  v(x,0) = \rho_0\,, \quad x\in\Omega\,,
 \end{equation} 
 describes the evolution of the compressible gas dynamics system  \eqref{gas2} along particle paths.

 Our main theorem for the system  \eqref{gas2} is the following. 
 
\begin{theorem}\label{th:energy}
Let $(\eta, w, s,v)\in C([0, T); BV^4 )$ be the admissible weak solution to \eqref{sys2} - \eqref{eq:ini2}.  The distributional solution $\gamma$ to 
 \begin{equation}
\gamma_x(x,t) = \eta(x,t), \quad \gamma _t(x,t) = \frac{w(x,t)}{ v (x)}
\end{equation} is well-defined, absolutely continuous, 
and invertible, and its inverse is absolutely continuous.
Define 
\begin{equation}\label{eq:rhoupE}
\rho := \frac{ v }{\eta }\circ \gamma^{-1}, \  u :=  \frac{w }{ v  } \circ \gamma^{-1}, \mbox{ and } p :=    \frac{\alpha-1}{\eta}( s - \frac12  \frac{w^2}{ v } )  \circ \gamma^{-1}.
\end{equation}
 Then $(\rho,u,p)$ is an admissible solution to the Cauchy problem for the compressible gas dynamics equations \eqref{gas2}.
\end{theorem}
In the proofs of Theorems \ref{th:energy} and \ref{th:mainentro} we will make use of the following change of variables lemma whose proof follows from the Radon-Nikod\'ym theorem.
\begin{lemma} \label{lemma2}
Let $A$ and $B$ be two functions of bounded total variation, let $\chi$ be a smooth test function and $\widetilde \chi  = \chi \circ \gamma$.   Then 
\begin{align*}  
\iint \chi _t A  \circ \gamma^{-1}  +\chi_x B  \circ \gamma^{-1}\, dx\,dt 
= 
\iint     \left(\widetilde \chi _t - \widetilde \chi_x  \frac{w }{\eta v } \right) A  \eta +\widetilde \chi_x B   \, dx\,dt 
\end{align*}

\end{lemma}

\begin{proof}[Proof of Theorem \ref{th:energy}]
We begin by considering the weak formulation of  system \eqref{gas2}:
\begin{align*} 
&I_1 = \iint \psi _t \rho +\psi_x\rho u \,  dx\,dt\\
&I_2= \iint \phi _t\rho u+\phi_x(\rho u^2+p) \,  dx\,dt\\
&I_3= \iint \chi _t \left(\frac12 \rho u^2+\frac{1}{\alpha-1} p  \right)+\chi_x\left(\frac12 \rho u^3+ \frac{1}{\alpha-1} pu +  p u \right) \,  dx\,dt. 
\end{align*}
We will show the details for $I_3$ since $I_1$ and $I_2$ are similar. We substitute the definitions in \eqref{eq:rhoupE}:
 
\begin{align*} 
&I_3= \iint \chi _t \left( \frac{s}{\eta}  \right)   \circ \gamma^{-1} 
 +\chi_x\left ( 
\alpha  \frac{sw}{\eta v} 
  + \frac{1-\alpha}2  \frac{w^3}{\eta v^2}  \right) \circ \gamma^{-1}\, dx\,dt.
\end{align*}
We apply Lemma \ref{lemma2} and simplify to find
\begin{align*} 
&I_3= \iint  \widetilde\chi _t s 
+ \widetilde \chi_x \left ( 
(\alpha  -1) \frac{sw}{\eta v} 
  + \frac{1-\alpha}2  \frac{w^3}{\eta v^2}  \right) \circ \gamma^{-1}\, dx\,dt.
\end{align*}
We now use the fact that $(\eta,w,s,v)$ is the weak solution to \eqref{sys2} - \eqref{eq:ini2}, to find 
\begin{align*} 
&I_3 = -\int \widetilde \chi (x,0) s(x,0) \, dx = -\int \widetilde \chi (x,0)
 \left( \frac{1}{2}\rho_0 u_0^2+ \frac{p_0}{\alpha -1}\right) \, dx\,.
\end{align*}
Hence, $(\rho, u, p)$ is a weak solutions to the system \eqref{gas2} as claimed. 

We now check that if $(\eta,w,s,v)$ is admissible as a solution of \eqref{sys2} 
then $(\rho, u, p)$ as defined in \eqref{eq:rhoupE} is admissible as a solution of \eqref{gas2}.
Assume we have, from Proposition \ref{entropy-condition}, a convex extension 
 $\mathcal E = \rho X(p\rho^{-\alpha})$, $\mathcal Q = u \rho X(p \rho^{-\alpha})$. 
 We will show that in the variables $(\eta, w, s,  v )$, 
 these correspond to convex extensions for the original system. 
 Thus, for a test function $\psi$, we assume that 
\begin{align*}
\iint &\partial _t\psi  \left[ \frac{ v }{\eta } X\left( \frac{\alpha-1}{\eta}\left( s - \frac12  \frac{w^2}{ v } \right) \left(\frac{ v }{\eta } \right)^{-\alpha}\right)
\right] \circ \gamma^{-1} 
   \\ & +
\partial _x \psi   \left[  \frac{w }{\eta }   X\left( \frac{\alpha-1}{\eta}\left( s - \frac12  \frac{w^2}{ v } \right) \left(\frac{ v }{\eta } \right)^{-\alpha}\right)
\right] \circ \gamma^{-1}  dxdt \le 0 \,.
\end{align*}
 Then, applying Lemma \ref{lemma2} we have 
\begin{align*}
\iint &\left(\widetilde \psi _t - \widetilde \psi_x  \frac{w }{\eta v } \right)  \left[ \frac{ v }{\eta } X\left( \frac{\alpha-1}{\eta}\left( s - \frac12  \frac{w^2}{ v } \right) \left(\frac{ v }{\eta } \right)^{-\alpha}\right)
\right] \eta
   \\ & +
 \widetilde\psi  _x \left[  \frac{w }{\eta }   X\left( \frac{\alpha-1}{\eta}\left( s - \frac12  \frac{w^2}{ v } \right) \left(\frac{ v }{\eta } \right)^{-\alpha}\right)
\right]  dxdt \le 0 \,.
\end{align*} 
It suffices to show that 
$$
\tilde{ \mathcal E} : =   v  X\left( \frac{1}{\eta}\left( s - \frac12  \frac{w^2}{ v } \right) \left(\frac{ v }{\eta } \right)^{-\alpha}\right) 
$$
is a convex function of the four variables $ v , \eta, w, s$. We re-write this as 
$$
\tilde{ \mathcal E} : =   v  X\left(  \eta^{\alpha-1} \left( \frac{  s}{ v ^{\alpha}} -  \frac{  w^2}{2 v ^{\alpha+1} }  \right)   \right) .
$$

Now assume that in the variables  $v,  w, s $, the function $ \mathcal E( v , w, s) =  v  X\left(   \frac{  s}{ v ^{\alpha}} -  \frac{  w^2}{2 v ^{\alpha+1} }  \right)   $ is  convex. Then, it can be expressed as the supremum of all affine functions which lie below it. Let $\mathcal L$ be the set of all affine functions which lie below $ \mathcal E $, let $\mathcal I$ be an index of the elements of domain of $\mathcal L$ and for each $i\in\mathcal I$,  let $\ell_i \in \mathcal L$ be given by 
$$
\ell_i = c_{0,i}+ c_{1,i}  v +c_{2,i}s+ c_{3,i}w
$$
Then, 
$$
 \mathcal E ( v , w, s)  =  v  X\left(   \frac{  s}{ v ^{\alpha}} -  \frac{  w^2}{2 v ^{\alpha+1} }  \right)    
= \sup_{ i \in \mathcal I} \left \{ 
c_{0,i}+ c_{1,i}  v +c_{2,i}s+ c_{3,i}w
\right\}.
$$
Now consider $ \tilde{ \mathcal E} $
\begin{align*}
 \tilde{ \mathcal E}  ( \eta, w, s, v)  =  v  X\left(  \eta^{\alpha-1} \left( \frac{  s}{ v ^{\alpha}} -  \frac{  w^2}{2 v ^{\alpha+1} }  \right)   \right)
& = \eta \mathcal E(\frac{ v }{\eta}, \frac{w}{\eta}, \frac{s}{\eta})
\\ & =
\eta \sup_{ i \in \mathcal I}\left \{ 
c_{0,i}+ c_{1,i}  \frac{ v  }{\eta}+c_{2,i}\frac{s}{\eta} + c_{3,i} \frac{w}{\eta} 
\right\}
\\
& = 
  \sup_{ i \in \mathcal I} \left \{ 
c_{0,i} \eta+ c_{1,i}  v +c_{2,i}s+ c_{3,i}w
\right\} \,,
\end{align*}
Thus $ \tilde{ \mathcal E} $ is   characterized as the supremum over a set of convex functions and is therefore convex. The above argument is reversible, and therefore, $( \eta, w, s, v) $ are admissible solutions to \eqref{sys2} if and only if $(\rho, u, p)$ are admissible solutions to \eqref{gas2}. 
\end{proof}
 
\subsection{Entropy  conserving solutions}
We now  construct a first order system that captures the flow of the system  \eqref{gas} along particle paths. We define $w=\rho_0 \gamma_t$,   $\eta=\gamma_x$, $r =  \rho_0 S_0$  and  we let  $v=\rho_0$.
Then the system for $(\eta, w, r, v)$ is
\begin{equation}\label{sys1}
\begin{split}
\eta_t-\left(\frac{w}{ v }\right)_x&=0\,,\\ 
w_t+\left(\frac{e^{r /  v  }  v ^{\alpha}}{\eta^{\alpha}}\right)_x&=0\,, \\ 
r_t & = 0\,,
\\ 
v_t & = 0 \,.
\end{split}
\end{equation}
 This system is paired with the initial data 
\begin{align}
 \eta(x,0) = 1\,,\;\;  w(x,0) =  \rho_0  u_0\,,\;\;  r(x,0)= \rho_0  S_0 \,, \text{  and }  v(x,0) =\rho_ 0\,,\quad x\in\Omega\,.
\end{align} 

\begin{theorem}\label{th:mainentro}
Let $(\eta, w,  r,v)$ be an admissible solution to \eqref{sys1}. The distributional solution $\gamma$ to 
 \begin{equation}
\gamma_x(x,t) = \eta(x,t), \quad \gamma _t(x,t) = \frac{w(x,t)}{ v (x)}
\end{equation} is well-defined, absolutely continuous, 
and invertible, and its inverse is absolutely continuous.
 Define 
\begin{align}
 \rho := \frac{  v }{\eta }\circ \gamma^{-1}  , \quad u :=  \frac{w  }{ v   }  \circ \gamma^{-1}  , \quad \text{ and } p :=  \frac{e^r v ^{\alpha} }{ \eta^{\alpha}}  \circ \gamma^{-1}.
\end{align} 
Then $(\rho,u,p)$ is a weak admissible solution to the Cauchy problem for system \eqref{gas}.
\end{theorem}

\begin{proof}
The proof that $(\rho,u,p)$ is a weak solution to \eqref{gas} is similar to the 
proof of Theorem 3, and we omit the details. 
Now we establish that, if $\rho, u$  and $p$ are as defined in the   theorem, then they are the admissible solutions to the Cauchy problem corresponding to system  \eqref{gas} - \eqref{eq:initial}.

From Proposition \ref{entropy-condition} we know that (other than the energy equation) any 
convex extension is given by 
\begin{equation}
\mathcal E = \rho X\left( p \rho^{-\alpha} \right)\,, \quad \mathcal Q = \rho u X\left(p \rho^{-\alpha}\right)\,.
\end{equation}
 We write the arguments of $\mathcal E$ in terms of the conserved quantities $\rho,$ $m = \rho u$ and $\sigma = \rho S$
\begin{equation}
\mathcal E = \rho X\left(e^{\sigma /\rho}\right)  = \rho Y\left(\frac{\sigma}\rho\right)\,,
\end{equation}
where  $Y$ is the composition of $X$ and the exponential. 
Thus, $\mathcal E$ is convex if the Hessian, $H ( \mathcal E)$, is positive definite, where
\begin{align*}
H(\mathcal E) 
=  \begin{bmatrix}
\frac{\sigma^2}{\rho^3} Y ''& \frac{-\sigma}{\rho^2} Y'' 
\\
\frac{-\sigma}{\rho^2} Y''  &\frac{1}{\rho} Y''
\end{bmatrix}\,,
\end{align*}
and it is easy to see that for any $a,b \in \rr$ we have 
\begin{align*}
(a,b) H(\mathcal E) \begin{bmatrix}
a\\b
\end{bmatrix}
=  \frac{1}{\rho}\left(\frac{a \sigma}{\rho} - b\right)^2 Y''\,.
\end{align*}
Thus, if $Y'' (x)= X''(x) e^{2x}+X' e^x>0$, the above Hessian is positive semidefinite and $\mathcal E$ is convex. 

Next, we apply Lemma \ref{lemma2} to find that $\mathcal E$ transforms to
$$
\tilde {\mathcal E} =  v   X \left( e^{r/ v }  \right)  =  v  Y\left(\frac{r}{ v } \right)
$$
iin the new coordinates $(\eta, w, r, v)$.  We see that $\tilde {\mathcal E} ( v ,r) =   {\mathcal E}( v , r)$ and therefore, $\tilde {\mathcal E} $ is convex if and only if $  {\mathcal E} $ is convex. 
\end{proof}

\begin{remark}\label{remark}
There are a number of ways of writing \eqref{divform} as a first-order system. 
For instance one way is to define new variables $\eta=\gamma_x$, $w=\rho_0\gamma_t$ 
and  $r=p_0$, mimicking the isentropic case.
Then the first order system for $(\eta, w, r)$ is
\begin{equation}
\begin{split}
\eta_t-\left(\frac{w}{\rho_0}\right)_x&=0\,,\\ 
w_t+\left(\frac{r}{\eta^{\alpha}}\right)_x&=0\,, \\
r_t & = 0\,.
\end{split}
\end{equation}
However this system does not lead to admissible weak solutions for either system \eqref{gas2} or 
\eqref{gas} even though  it evolves along the particle paths as long as the solutions are classical. 
\end{remark}

\section{Scalar Convex Conservation Laws} \label{secfour}

The inviscid  Burgers equation, $\rho_t+\rho\rho_x=0$, is the best-known example
of a convex scalar conservation law. 
Constantin and
Kolev \cite{CoKo} found interesting geometric significance to trajectories
$\gamma_t=u\circ\gamma$ for the scaled equation $u_t+3uu_x=0$.
In fact, these trajectories are geodesic equations through the group of diffeomorphisms, but do not correspond either to characteristics or to our notion
of particle paths developed in this section, and they do not generalize to paths for weak solutions. To develop a theory of particle paths for Burgers equation that extends to weak solutions, we require a more suitable notion of velocity. 

A prototype equation for traffic flow on a one-way
road,
\begin{equation} \label{traffic}
\rho_t+\big(\rho(1-\rho)\big)_x=0\,
\end{equation}
provides some intuition.
 Lighthill and Whitham \cite{LiWh} and Richards \cite{R1956}
developed the continuum model, $\rho_t+(\rho u(\rho))_x=0$, for one-way traffic flow. 
Here $\rho$ represents the linear density of vehicular
traffic, and $u(\rho)$ is the velocity, assumed to be a strictly decreasing
function of the density.
Equation \eqref{traffic} is a special case of the Lighthill-Whitham-Richards model,
with velocity $u=1-\rho$ a linear function of density, normalized so that $u=1$ represents the maximum
speed and $\rho=1$ the maximum density, at which the road is saturated
and traffic is at a standstill.
Since conservation of $\rho$ is equivalent to conservation of $u=1-\rho$ in
this case, one can rewrite the equation as a conservation law for $u$:
$$ u_t+\big(u(u-1)\big)_x=0\,.$$
Further scaling recasts the data as $\rho_0$ or $u_0$ with
$0<\rho_0<1$.
The diffeomorphism formulation cannot handle  $\rho=0$, corresponding to zero density, or vacuum, but a linear scaling allows us to avoid
this difficulty.
A similar problem arises for the compressible flow equations, where we did not
allow vacuum.
There, avoiding a vacuum is an essential restriction on the data.
The particle path formulation follows the trajectory of an
individual car through the changing configuration of the traffic flow on the road.
Periodic data represents a closed loop, for example a racetrack.
Shock waves, or almost instantaneous slowdowns
on the highway are also familiar in daily life, and the phenomenon
of navigating one's way through them provides an example of the weak particle
paths defined in this paper.

It was this example that motivated our quest to interpret solutions
of Burgers equation as diffeomorphisms 
in a way that extends to weak solutions.
In some respects, the example of a scalar equation is less intuitive than the
compressible gas dynamics equations, because an equation like Burgers equation is not usually described in terms of mass fluxes.
 
With traffic flow as motivation, we identify a particle path formulation for any scalar convex conservation law
 in a single space variable,
\begin{equation} \label{cons1}
\rho_t+f(\rho)_x=0\,,
\end{equation}
with $f''>0$.
 If $\rho(x,t)$ is a mass distribution moving in a channel, then
 $f$ is the specific flux and can be written as
 $f=\rho u$, where $u$ is defined as
 \begin{equation} \label{consu}
 u\equiv \frac{f(\rho)}{\rho}\equiv F(\rho)\,; \end{equation}
 $u$ represents the velocity of the `particle' at $(x,t)$ with density $\rho$. 
 In the case of Burgers equation, $f(\rho) = \frac12 \rho^2$;
 we note that $u=\rho/2$; that is, $u$ is not the characteristic speed. 
 To avoid some technical complications in arriving at a diffeomorphism, 
 we would like $F$ and $\rho$
  to be positive.
  This can be achieved by adding
 a constant $C$ to $\rho$ and
 a linear multiple, $m\rho$, to $f$
 (or, what amounts to the same thing, by mapping $x\mapsto x+ct$
 for an appropriate value of $c$). 
Since weak solutions of \eqref{cons1} are bounded in $L^\infty$
 by the bounds on the initial data, the choices of $C$, $m$ 
 and $c$ are specific to given Cauchy data.
 
 Equation \eqref{cons1} now takes the form $\rho_t+(\rho u)_x=0$, where $u$
 is a function of $\rho$.
We want to think of $\rho$ as a function of $u$,
 say
 $$ \rho = g(u)\equiv F^{-1}(u)\,.$$
 For this, we need \eqref{consu} to be invertible, that is
 $$F'(\rho)= \frac{d}{d\rho}\left(\frac{f(\rho)}{\rho}\right)=\frac{\rho f'-f}{\rho^2}>0\,,$$
 which is clearly true for any super-linear function, and is a consequence
 of strict convexity.
 It is at this point that we require the condition that the original flux
 function, $f$, be convex.
 We now express the original equation \eqref{cons1} in terms of $u$:
 \begin{equation} \label{consg}
 \big(g(u)\big)_t+\big(ug(u)\big)_x=g'(u)u_t+ug'(u)u_x+g(u)u_x=0\,,
\end{equation}
noting that $\rho=g(u)$ is the correct conserved quantity.

A point $x$ at $t=0$ can be mapped to the position $\gamma(x,t)$
it would reach at time $t$  if at every point it traveled at the velocity  $u$ associated
with the solution to conservation law at that point.
That is, $\gamma$ satisfies the equation
\begin{equation} \label{gammadef}
 \gamma_t=u(\gamma(x,t),t)=u\circ\gamma\,, \quad \gamma(x,0)=x \,.\end{equation}
As long as  $u(\cdot,t)$ is locally Lipshitz continuous, this is can be solved as an ordinary differential equation.
 However, we can use \eqref{gammadef} to find a partial differential equation,
in fact a conservation law, which allows us to
 define {\em weak} paths, 
 particle paths which are absolutely continuous 
 but not differentiable and do not require $u$ to be continuous. 

Our main result in this section is that the mapping $x\mapsto \gamma(x,t)$
defines for each fixed $t$ a diffeomorphism
for smooth solutions, and generalizes to a `weak diffeomophism' when
classical solutions no longer exist.
As we did in Sections \ref{sectwo} and \ref{secthree} for gas dynamics,
we find an equation for $\gamma$ and
show that it has weak solutions that are equivalent to the weak solutions 
of \eqref{cons1}.

To derive the diffeomorphism equation for smooth solutions,
simplify \eqref{consg}  to
$$ u_t+uu_x+\frac{g(u)}{g'(u)} u_x=0\,.$$
and differentiate to obtain
$$ \gamma_{tt}=(u_t+uu_x)\circ\gamma = -\left(\frac{g(u)}{g'(u)}u_x\right)\circ\gamma\,.$$
Eliminate $u_x$ from this relation by differentiating \eqref{gammadef}
with respect to $x$:
$$ u_x=({\gamma}_t\circ\gamma^{-1})_x=({\gamma}_{xt}\circ\gamma^{-1})
\cdot \partial_x\gamma^{-1}= \left(\frac{\gamma_{xt}}{\gamma_x}\right)\circ\gamma^{-1}\,,
$$ 
so that
$$ \gamma_{tt}=-\left(\frac{\gamma_{xt}}{\gamma_x}\right)
\left(\frac{g(u)}{g'(u)}\right)\circ\gamma
=-\left(\frac{\gamma_{xt}}{\gamma_x}\right)
\left(\frac{g(\gamma_t)}{g'(\gamma_t)}\right)\,.
$$
Write this as 
$$
\left(\frac{g'(\gamma_t)}{g(\gamma_t)}\right)\gamma_{tt}=-\frac{\gamma_{xt}}{\gamma_x}
$$
and integrate from $t=0$ to $t$, noting that $\gamma_t(x,0)=u_0(x)=F(\rho_0(x))$,
and that all quantities are positive, so
$$ \log \left(\frac{g(\gamma_t)}{g(\gamma_t(0))}\right)=
-\log\left(\frac{\gamma_{x}}{\gamma_x(0)}\right)\quad\textrm{or}\quad
 \log \left(\frac{g(\gamma_t)}{\rho_0}\right)=
-\log\gamma_{x}\,,
$$
since $\gamma_x(x,0)\equiv 1$ and $g(F(\rho_0))=\rho_0$.
Finally,
$$ \frac{g(\gamma_t)}{\rho_0}=\frac{1}{\gamma_x}\,,$$
which we invert to obtain
\begin{equation} \label{consgamma}
\gamma_t=F\left(\frac{\rho_0}{\gamma_x}\right)\,.
\end{equation}
This is a Hamilton-Jacobi equation for $\gamma$.
As a Hamilton-Jacobi equation it has a rather awkward structure,
but we convert it to a
conservation law system by defining $\eta\equiv\gamma_x$ and 
$v\equiv \rho_0$ to obtain
\begin{equation} \label{conseqn}
\begin{split} \eta_t-F\left(\frac{v}{\eta}\right)_x&=0\\
v_t&=0\,. 
\end{split}\end{equation}
 While the
scalar conservation law
\begin{equation}\label{vcons} 
\ \eta_t-F\left(\frac{\rho_0}{\eta}\right)_x=0\,,\end{equation}
obtained by differentiating \eqref{consgamma} with respect to $x$ 
might seem preferable to the system \eqref{conseqn}, equation \eqref{vcons} has some
disadvantages, as
it embeds the initial condition in the equation, destroying the attractive
semigroup property.
In addition, while there is extensive theory, dating back to Kru\v{z}kov \cite{Kr},
for conservation laws of the form $u_t+f(u,x)_x=0$, those results generally require
that $f$ be differentiable in $x$ (see \cite{Dafbook}).
Therefore, system  \eqref{conseqn}  is a more suitable system from which to construct particle paths. 
 
Indeed, we now show  that there is a one-to-one correspondence between convex extensions
for the scalar conservation law \eqref{cons1} and convex extensions
for \eqref{conseqn} from
which the diffeomorphisms are constructed, and that admissible weak  solutions of one
system correspond to admissible weak solutions of the other.
For the scalar equation \eqref{cons1}, any convex function $\E(\rho)$ of $\rho$ provides a
convex extension, with corresponding flux $\Q$ obtained from
\begin{equation} \label{qentdef}
 \Q'(\rho)= \E'(\rho)f'(\rho)\,.\end{equation}
For the system \eqref{conseqn} any convex extension is of the form
\begin{equation} \label{scentropy}
\wE = \eta X\left(\frac{v}{\eta}\right)\,,\quad \wQ = \wQ\left(\frac{v}{\eta}\right)=
 -\int^{v/\eta}\big(X(x)-xX'(x)\big)F'(x)\,dx\,,\end{equation}
for $X$ a function of a single variable with $X''>0$.
(The conclusion on convexity is easily checked by calculating the Hessian
of $\wE$.) 
 
Additionally, \eqref{conseqn}
is a particularly simple example of a Temple
system, for which global large-data solutions exist; see Leveque and Temple \cite{LT00}.
We quote the result here.
\begin{proposition} \label{exist1}
Provided that $F\in C^2$ is positive and that $\rho_0$ is of bounded variation, 
the Cauchy problem for  system \eqref{conseqn}
with initial  data
\begin{equation} \label{scalardata}
\eta(x,0)= 1\,,\quad v(x,0)=\rho_0(x) \geq m > 0\,.\end{equation} 
has a unique admissible weak solution for all time,
and $(\eta,v)\in C([0,\infty); BV^2 )$. 
\end{proposition}
\begin{proof} 
  Temple systems have the property that different characteristic families
  do not interact in a nonlinear way.
  In this example,
one family is genuinely nonlinear and the other linearly degenerate.
The characteristic speeds are
$$ \lambda_1=\frac{v}{\eta^2}F'(v/\eta) >0=\lambda_2\,,$$
and the corresponding right eigenvectors are
$$ {\mathbf r}_1=\left(\begin{array}{c}1\\0\end{array}\right)\,,\quad
{\mathbf r}_2=\left(\begin{array}{c}\eta\\v\end{array}\right)\,.
$$ 
Temple systems possess as positively invariant regions quadrilaterals
(in state space)
whose sides are parallel to right eigenvectors.
In this case, the sides are parallel to the $\eta$-axis (in the $\eta$-$v$
plane) or are segments of radial lines from the origin.
Given the initial conditions $\eta(x,0)=1$ and $v(x,0)=\rho_0(x)$, 
with $0<m\leq \rho_0\leq M$, say, it is
straightforward to see that $(\eta,v)$ lies inside the quadrilateral $\mathbf R$
with vertices $(m/M,m)$, $(1,m)$, $(M/m,M)$ and $(1,M)$, so
provided that the assumptions on $f$ (and hence $F$) hold in
this range, a unique admissible weak solution with range in $\mathbf R$
is defined for all $t>0$.  
The weak solutions are admissible for any of the convex extensions
defined in equation \eqref{scentropy}.
\end{proof}

We now prove that we can recover the admissible weak solutions to
the original conservation law \eqref{cons1} 
from \eqref{conseqn}.

\begin{theorem}\label{th1}
 Let $(\eta,v)\in C([0, \infty); BV^2)$ be the admissible weak solution to \eqref{conseqn}, 
with   Cauchy data  
\eqref{scalardata},
 with $(\eta,v)\in\mathbf R$.
The distributional solution $\gamma$ to 
 \begin{equation}\label{gamma9}
\gamma_x(x,t) = \eta(x,t), \quad \gamma _t(x,t) = F\left(\frac{v(x,t)}{\eta(x,t)}\right)
\end{equation} is well-defined, absolutely continuous, 
and invertible, and its inverse is absolutely continuous.
Define 
 \begin{equation}\label{rho}
\rho := \frac{v(\gamma^{-1}(x,t),t)}{\eta(\gamma^{-1}(x,t),t)}\,.\end{equation}
 Then $\rho$ is an admissible weak solution to the 
 Cauchy problem for the scalar convex conservation law \eqref{cons1}.
\end{theorem} 
\begin{proof} 
It is easy to see that \eqref{gamma9} is consistent, from the 
first equation of \eqref{conseqn}. 
As the antiderivative of a strictly positive function of bounded variation, 
$\gamma(\cdot, t)$ is absolutely continuous and invertible,
 and its inverse is absolutely continuous. 
 Since both $\gamma$ and $\gamma^{-1}$ are well defined, we 
 can define $\rho$ from \eqref{rho}. 
 We claim that $\rho$ is the admissible weak  solution to \eqref{cons1}.

Let $\varphi (x,t)$ be a test function and consider 
the weak form of \eqref{cons1}, noting that $f(\rho)=\rho F(\rho)$:
\begin{align*} 
I&=\iint \big[\varphi_t\rho +\varphi_xf(\rho)\big]\,dx\,dt
\\
&=\iint \left[ \varphi_t\left(\frac{v}{\eta}\right)\circ\gamma^{-1}
+\varphi_x\left(\frac{v}{\eta}F\left(\frac{v}{\eta}\right)\right)
\circ\gamma^{-1}\right]\,dx\,dt\,.
\end{align*}
Since for almost every $t$, $\gamma$ is a strictly increasing, 
absolutely continuous, and almost everywhere differentiable function,
 the Radon-Nikod\' ym theorem allows us to make the change of variables 
 $(x,t)\mapsto (\gamma (x,t),t)$.

Under this change of variables, 
$$I=\iint \left[(\varphi_t\circ\gamma) \left(\frac{v}{\eta}\right)
+(\varphi_x\circ\gamma) \left(\frac{v}{\eta}F\left(\frac{v}{\eta}\right)\right)\right]\eta\,dx\,dt\,.
$$
Now, 
\begin{equation} \label{useful}
 \partial_t\big(\varphi(\gamma(x,t),t)\big)=\varphi_x(\gamma,t)\gamma_t
+\varphi_t(\gamma,t)\,,\quad \partial_x\big(\varphi(\gamma(x,t),t)\big)=
\varphi_x\gamma_x\,,\end{equation}
where the subscripts denote derivatives in the first or second variable.
Letting $\widetilde \varphi(x,t)=\varphi(\gamma(x,t),t)$, we have
$$\varphi_t\circ\gamma =\widetilde\varphi_t-
\frac{\widetilde \varphi_x}{\gamma_x}\gamma_t
\quad\textrm{and}\quad \varphi_x\circ\gamma=\frac{1}{\gamma_x}\widetilde \varphi_x\,,
$$
so, since $\gamma_x=\eta$, and using \eqref{gamma9} for $\gamma_t$,
\begin{align*}
I&=\iint \left[\left(
\widetilde \varphi_t-\frac{\widetilde \varphi_x}
{\gamma_x}\gamma_t\right)\left(\frac{v}{\eta}\right)
+\frac{\widetilde \varphi_x}{\gamma_x}\left(\frac{v}{\eta}
F\left(\frac{v}{\eta}\right)\right)\right]\eta\,dx\,dt\\
&=\iint\left[\widetilde\varphi_t v +
 \widetilde \varphi_x\left(\frac{v}{\eta}F\left(\frac{v}{\eta}\right)
-\frac{v}{\eta}F\left(\frac{v}{\eta}\right)\right)\right]\,dx\,dt\\
&=\iint\widetilde\varphi_t v \,dx\,dt=-\int \widetilde\varphi(x,0)\rho_0(x)\,dx\,,
\end{align*}
which shows that a weak solution to \eqref{conseqn} determines a
weak solution to \eqref{cons1}.

To complete the proof, we need to show that admissibility  conditions for
one system translate to admissibility conditions for the other.
We first show that given a convex extension for $\rho$, 
we can recover a convex extension for $(\eta,v)$, and then 
show the converse. 
Suppose that $(\E,\Q)$ is a convex extensionr for \eqref{cons1},
so that $\E$ and $\Q$ are related by equation \eqref{qentdef}.
Then if $\rho$ is an admissible weak solution to \eqref{cons1}, we have
\begin{equation} \label{ent1}
J\equiv \iint \big(\psi_t\E+\psi_x\Q\big)\,dx\,dt\geq 0\,.
\end{equation}
Apply a change of variables $(x,t)\mapsto(\gamma(x,t),t)$ so, as before,
$$J=\iint \psi_t(\gamma,t)\E\left(\frac{v}{\eta}\right)+
\psi_x(\gamma_t)\Q\left(\frac{v}{\eta}\right)\,\eta\,dx\,dt\,
$$
Now use \eqref{useful}, and let $\widetilde \psi$ denote the test
function in the new variables, to find   
$$ J=\iint\partial_t \widetilde \psi  \wE(\eta,v) + 
\partial_x\widetilde \psi  \wQ(\eta,v)\,dx\,dt\geq 0\,,$$
with
$$ \wE(\eta,v)=\eta \E\left(\frac{v}{\eta}\right)\,,\quad 
\wQ(\eta,v)=\Q\left(\frac{v}{\eta}\right)-
F\left(\frac{v}{\eta}\right)\E\left(\frac{v}{\eta}\right)\,.$$
Thus $\wE$ is of the form \eqref{scentropy} with $X(\cdot)=\E(\cdot)$
and one can verify that $\wE_t+\wQ_x=0$ 
whenever $(\eta,v)$ is
a classical solution to  \eqref{conseqn}. 
The Hessian matrix of $\wE$ is
positive semi-definite.
The opposite direction is similar, and
this completes the proof.
\end{proof} 
 

 \section{Systems of Two Equations: Working With Riemann Invariants}
 \label{secfive}
In this section we identify a class of systems of two equations for which a
Eulerian-Lagrangian correspondence,
analogous to the correspondence for isentropic gas dynamics, exists.
That is, even when  a system of two conservation laws does not have  physical particle
paths, under certain rather general conditions a quantity playing that role exists, and extends, 
as in gas dynamics, to weak solutions preserving the conserved quantities of the original system.

We write the system as $U_t+F_x\equiv U_t+(dF)U_x=0$, or
\begin{equation} \label{orig}
\begin{pmatrix} u\\v\end{pmatrix}_t + \begin{pmatrix} f(u,v)\\g(u,v)\end{pmatrix}_x = 
\begin{pmatrix} u\\v\end{pmatrix}_t + \begin{pmatrix} f_u&f_v\\g_u&g_v\end{pmatrix}
\begin{pmatrix} u\\v\end{pmatrix}_x = 
0\,.
\end{equation}
Riemann invariants, $(z,w)$, always exist  \cite{serre} for
a system of two equations;
$Z=(z,w)^T$ satisfies a diagonal system, not in conservation
form: $Z_t+\Lambda Z_x=0$.
Here $\Lambda=
\textrm{diag}(\lambda_1,\lambda_2)$; $\lambda_i$ are the characteristic directions --
the eigenvalues of $dF$.
 The components of $Z$ are functions of $U$
and are found by computing
\begin{equation} \label{rieqn}
 Z_t+\Lambda Z_x= (dZ)U_t+\Lambda(dZ)U_x=(dZ)(-(dF)U_x)+
\Lambda(dZ)U_x=0\,,
\end{equation} 
or
\begin{equation} \label{riform}
 (dZ)(dF)=\Lambda(dZ)\,,\end{equation}
so $\nabla z$ and $\nabla w$ are left eigenvectors of $dF$.
There are two independent Riemann invariants when $dF$ has distinct
eigenvalues, since then $\nabla z$ and $\nabla w$, the rows of $dZ$,
are linearly independent, and $Z$ can then replace $U$ as the dependent variable.
Equation \eqref{rieqn} is  meaningful only for classical solutions, and in this case, the solution corresponds to the classical solution of \eqref{orig}. 

Working with the system
\begin{equation} \label{RI} Z_t+\Lambda Z_x=
\begin{pmatrix}z\\w\end{pmatrix}_t +\begin{pmatrix}\lambda_1&0\\0&\lambda_2\end{pmatrix}
\begin{pmatrix}z\\w\end{pmatrix}_x=0\,,
\end{equation}
we write the eigenvalues as
\begin{equation} \label{eigs}
\lambda_{1,2} = T\mp E\,,
\textrm{ where } T=\frac{1}{2} {\textrm{Trace\,}}{dF}=
\frac{1}{2}\big(f_u+g_v\big)\,,\textrm{ and } E^2=T^2-D\,,
\end{equation}
where $D={\textrm{Det\,}}dF=f_ug_v-f_vg_u$, so
$E^2= \frac{1}{4}(f_u-g_v)^2+f_vg_u$.
Assuming the system is strictly hyperbolic, then $E^2>0$, and we take $E>0$.

The key point is the choice of a particle path.
As with a scalar equation, it should not coincide with a
characteristic speed, since characteristics intersect
in forward time, at the time that the solution ceases to be smooth,
and many times after that, every time a new shock forms.
The example of isentropic gas dynamics suggests a suitable candidate.
\begin{example}
 A formulation of the isentropic gas dynamics system \eqref{continuity} in $(u,v)$ coordinates, with $u$ the
density variable and $v$ the velocity, is
\begin{equation} \label{gasdyn}
\begin{pmatrix} u\\v\end{pmatrix}_t + \begin{pmatrix} uv\\\frac{1}{2}v^2
+q(u)\end{pmatrix}_x = 
0\,,
\end{equation}
where $q'(u)=p'(u)/u$.
The characteristic speeds are $v\mp c(u)$,
where  $c=\sqrt{p'(u)}$ is the local speed
of sound.
The Riemann invariants are $v\mp Q(u)$, with $Q'(u)=\sqrt{p'(u)}/u$.
In this example, $T=v$ and $E=\sqrt{p'(u)}$.
The velocity variable, $v$, defines the particle path by $\gamma_t=v\circ\gamma$.
\end{example}

Since for gas dynamics $T=v$ and $E=c$, it is reasonable to try to
generalize this by taking $T$ to be the particle path.
Other choices may be possible; $T$ has the advantage that it never coincides
with a characteristic speed in a strictly hyperbolic system.
We make the following assumption.
\begin{assumption} \label{assump}
{\em The mapping $(u,v) \mapsto (T,E)$ is smooth and smoothly invertible, and hence
the pair $(T,E)$ is an alternative choice for the state  variables.\/}
\end{assumption}
This is again the case for gas dynamics, since $c$ is an increasing
function of $u$ for a genuinely nonlinear system. It
is not the case for the Lagrangian
form of the gas dynamics equations.

\begin{example}
The gas dynamics system in Lagrangian coordinates is
$$ \tau_t-v_x=0\,,\quad v_t+p(\tau)_x=0\,,$$
where $\tau=1/\rho$ is the specific volume.
Now $T\equiv 0$ and $E=c$, so $(T,E)$ does not give a coordinate
system for the state variables.
Unsurprisingly, since the system is already in Lagrangian form, we cannot
transform it! 
In this case, one can construct the diffeomorphism
directly by integrating one of the variables.
The Riemann invariants may be chosen as $v\pm P(\tau)$, where $P'=\sqrt{p'}$.
The Riemann invariants form an alternative choice for state variables, but $T$ and $E$
do not serve this function.
\end{example}

\subsection{The Procedure} \label{procedure}
First convert the system
\begin{eqnarray} 
z_t+(T-E)z_x&=&0\\
w_t+(T+E)w_x&=&0 
\end{eqnarray}
into a system for $T$ and $E$: $z_t=z_TT_t+z_EE_t$ and so on, so
\begin{eqnarray*}
z_t&=&z_TT_t+z_EE_t=-(T-E)(z_TT_x+z_EE_x)\\
w_t&=&w_TT_t+w_EE_t=-(T+E)(w_TT_x+w_EE_x)\,.
\end{eqnarray*}
We obtain a system for $T$ and $E$.
Let $\Delta \equiv z_Tw_E-z_Ew_T$; then 
\begin{eqnarray}  \label{Teqn}
T_t+TT_x&=& \frac{E}{\Delta}\big((z_Tw_E+z_Ew_T)T_x+2z_Ew_EE_x\big)\\
E_t+TE_x&=&-\frac{E}{\Delta}\big(2z_Tw_TT_x+(z_Ew_T+z_Tw_E)E_x\big)\,.
\label{Eeqn}
\end{eqnarray}
Since we have assumed that  the mapping $(T,E)\mapsto(z,w)$ is invertible, we have $\Delta \neq 0$.

The terms that appear in \eqref{Teqn} and \eqref{Eeqn},
$\Delta$ and the coefficients of $T_x$ and $E_x$,
are given functions of $T$ and $E$.
We now define a particle path based on $T$, in the usual way; 
use $E$ to define a second variable,
\begin{equation}\gamma_t=T\circ\gamma\,,\quad\textrm{and}\quad \zeta =E\circ\gamma\,,
\end{equation}
and compute
$$ \zeta_t=E_t\circ\gamma +(E_x\circ\gamma)\gamma_t=(E_t+TE_x)\circ\gamma\,.$$
The relations
\begin{equation} \label{rels}
 T_x\circ\gamma = \frac{\gamma_{tx}}{\gamma_x}\quad\textrm{and}
\quad E_x\circ\gamma =\frac{\zeta_x}{\gamma_x}
\end{equation}
come from differentiating $T\circ\gamma=\gamma_t$ 
and $\zeta=E\circ\gamma$ with respect to $x$.

Now, composing \eqref{Teqn} and \eqref{Eeqn} with $\gamma$, we obtain 
\begin{align*}
\gamma_{tt}=(T_t+TT_x)\circ\gamma&= \left\{\frac{E}{\Delta}\bigg((z_Tw_E+z_Ew_T)T_x+2z_Ew_EE_x\bigg)\right\}\circ\gamma\\\
 \zeta_t=(E_t+TE_x)\circ\gamma &= 
\left\{-\frac{E}{\Delta}\bigg(2z_Tw_TT_x+(z_Ew_T+z_Tw_E)E_x\bigg)\right\}\circ\gamma\,.
\end{align*}
Define
 \begin{equation}\label{earlier}
  B= \left(\frac{z_Tw_E+z_Ew_T}{z_Tw_E-z_Ew_T}\right)\circ\gamma\,,\;
  C= \left(\frac{2z_Ew_E}{z_Tw_E-z_Ew_T}\right)\circ\gamma\,,\;
   D= \left(\frac{-2z_Tw_T}{z_Tw_E-z_Ew_T}\right)\circ\gamma\,;
 \end{equation}
 these are functions of $\gamma_t$ and $\zeta$.
Using \eqref{rels},  we obtain a pair of equations for $\gamma$ and $\zeta$
$$
\gamma_{tt}= \zeta B\frac{\gamma_{tx}}{\gamma_x}+\zeta C\frac{\zeta_x}{\gamma_x}\,,\quad
 \zeta_t = \zeta D\frac{\gamma_{tx}}{\gamma_x}-\zeta B\frac{\zeta_x}{\gamma_x} \,.
$$
We can express this as a first-order system by defining 
 $\eta\equiv\gamma_x$ and $\xi=\gamma_t$; then $B$, $C$ and $D$ are functions of $\xi$ and
 $\zeta$.
 We have initial conditions for all three variables:
\begin{alignat}{2} 
 \xi_t&=\frac{\zeta}{\eta}\bigg(B(\xi,\zeta) \xi_x+C(\xi,\zeta)\zeta_x\bigg)\,,
  \quad &\xi(x,0)&=\gamma_t(x,0)
 =T(x,0)=T(U(x,0))\,,\notag \\
 \zeta_t&=\frac{\zeta}{\eta}\bigg(D(\xi,\zeta) \xi_x-B(\xi,\zeta)\zeta_x\bigg)\,,
 \qquad &\zeta(x,0)&=\left.E\circ\gamma
 \right|_{t=0}=E(U(x,0))\,, \label{moregensys}\\
\eta_t&=\gamma_{xt}=\xi_x\,,\qquad\quad &\eta(x,0)&=\gamma_x(x,0)\equiv 1\,.\notag
 \end{alignat}
This quasilinear system has the form $\Xi_t=\mathcal A(\Xi) \Xi_x$ with $\Xi=(\xi,\zeta,\eta)^T$
and
$$\mathcal  A=
\frac{\zeta}{\eta}
\begin{pmatrix}B&C &0\\ D&-B&0\\1&0&0\end{pmatrix}\,.$$
This system is not in conservation form, but we can resolve this difficulty,

\begin{proposition} \label{gen-case}
The system \eqref{moregensys} can be replaced by a system in conservation form,
equivalent to \eqref{moregensys} for smooth solutions.
Furthermore, there exist
conserved quantities $h$ of the form $\eta \bar{h}(\xi,\zeta)$,
corresponding to conserved quantities $\bar{h}\circ\gamma^{-1}$ 
in the original 
conservation law system.
\end{proposition}

\begin{proof}
If there are functions $h$ and $k$ of $(\xi,\zeta,\eta)$ with $h_t+k_x=0$ then
\begin{align*}
0=&h_\xi\xi_t+h_\zeta\zeta_t+h_\eta\eta_t+k_\xi\xi_x+k_\zeta\zeta_x+k_\eta\eta_x\\
=&h_\xi\frac{\zeta}{\eta}\left(B\xi_x+C\zeta_x\right)+h_\zeta\frac{\zeta}{\eta}\left(
D\xi_x-B\zeta_x\right)+
h_\eta\xi_x+k_\xi\xi_x+k_\zeta\zeta_x+k_\eta\eta_x\,.
\end{align*}
This is an identity when $h$ and $k$ are solutions of the system
\begin{align*}
\frac{\zeta }{\eta}\left(Bh_\xi+Dh_\zeta\right)+h_\eta +k_\xi&=0 \notag\\
\frac{\zeta}{\eta }\left(Ch_\xi-Bh_\zeta\right)+k_\zeta&=0  \label{consys1}\\
k_\eta&=0\,. \notag
\end{align*}
From the third equation, $k=k(\xi,\zeta)$. Equating the expressions for $k_{\xi\zeta}$
in the first two equations gives a second-order equation for $h$:
\begin{equation*} \label{conscond}
\left[\frac{\zeta }{\eta}\left(Bh_\xi+Dh_\zeta\right)+h_\eta\right]_\zeta
=\left[ \frac{\zeta}{\eta }\left(Ch_\xi-Bh_\zeta\right)\right]_\xi\,.\end{equation*}
If we seek solutions of the form $h=\eta\bar h(\xi,\zeta)$, the equation simplifies to
\begin{equation} \label{hbareq}
\left[ {\zeta}\left(C\bar h_\xi-B\bar h_\zeta\right)\right]_\xi
-\left[{\zeta }\left(B\bar h_\xi+D\bar h_\zeta\right)+\bar h\right]_\zeta=0\,.
\end{equation}
The characteristic form, after dividing by $\zeta$, which is positive,  is
\begin{equation} \label{charform} C\bar h_{\xi\xi} -2B\bar h_{\xi\zeta} -D\bar h_{\zeta\zeta}\,.
\end{equation}
In Lemma \ref{nextlem} below we show
that the discriminant of equation \eqref{charform}, $B^2+CD$, is unity.
  Hence equation \eqref{hbareq} is hyperbolic, and any two functionally independent solutions allow us to replace 
\eqref{moregensys} by a system in conservative form, since the third equation in \eqref{moregensys} is already
conservative.
\end{proof}
\begin{lemma} \label{nextlem}
The discriminant  $B^2+CD=1$.
\end{lemma}
\begin{proof}
At least one of $z_T$ and $w_T$ is non-zero, since the mapping 
$(T,E)\mapsto(z,w)$ is assumed to be invertible. 
If we assume that both are non-zero and abbreviate the ratios $z_E/z_T=z_R$, $w_E/w_T=w_R$,
 then we have
 $$ B=\frac{w_R+z_R}{w_R-z_R}\circ\gamma\,; \quad C=\frac{2w_Rz_R}{w_R-z_R}\circ\gamma\,; \quad
 D=\frac{-2}{w_R-z_R}\circ\gamma\,,
 $$
 so 
 $$ B^2+CD=\frac{(w_R+z_R)^2+2z_Rw_R(-2)}{(w_R-z_R)^2}\circ\gamma=1\,.$$
If one of $z_T$ and $w_T$ is zero then
the calculation simplifies; the result is the same.
\end{proof}
{Let $h_1$ and $h_2$ be two conserved quantities found by this argument.
If $h_1=\eta u\circ\gamma$ and $h_2=\eta v\circ\gamma$, where $u$ and $v$ are the original conserved quantities in \eqref{orig}, 
we now have a conservation law system of the form
\begin{align}
\partial_t  h_1   + \partial_x k_1(h_1, h_2, \eta) &= 0 \notag \\
\partial_t    h_2   + \partial_x k_2(h_1, h_2, \eta) &= 0 \label{hsyst} \\ 
\partial_t \eta  + \partial_x k_3(h_1, h_2, \eta) &= 0 \,.\notag
\end{align}
The functions $k_i$ are defined as functions of $\xi$ and $\zeta$, but by our assumption of mutual invertibility of the pairs $(u,v)$, $(z,w)$ and $(T,E)$ or $(\xi,\zeta)$, they can we written as functions of $u\circ\gamma$ and $v\circ\gamma$, that is, of $h_1$ and $h_2$.
Assume that $h_1, h_2,$ and $\eta$ are weak solutions of \eqref{hsyst}. Our goal now is to construct the particle path, $\gamma$, then construct solutions to the original equations by setting $ u = \frac{h_1}\eta \circ \gamma^{-1} $ and $ v = \frac{h_2}\eta \circ \gamma^{-1} $. It is easy to see that $\gamma _x = \eta$ and we know that  $\gamma_t = \xi$. Therefore, we use Proposition 5 to formally solve $\bar h_1(\xi, \zeta)$ and  $\bar h_2(\xi, \zeta)$ for $\xi$ and $\zeta$. This yields $\xi(h_1, h_2) $. 
We let $\gamma$ be the distributional solution to 
$$
\gamma _x = \eta, \quad \gamma_t = \xi(h_1, h_2). 
$$
Since $\eta$ is BV, $\gamma$ is absolutely continuous in the spatial variable and continuous in time. 
By taking data sufficiently close to constant in the total variation norm, we can ensure that
 $\eta$ is strictly positive, and hence $\gamma$ is invertible. Thus, $\gamma^{-1}(x,t)$ is well defined on  $[0, \infty)$.  This gives us a particle path corresponding to the Lagrangian type formulation above. 

We have not identified $k_1, k_2,$ or $k_3$ as functions of $h_1, h_2$ and $\eta$. Nor have we explicitly found $\xi$ as a function of $h_1$ and $h_2$. The next subsection shows that this is possible. We return to the example of gas dynamics to illustrate how this is done for a particular system of equations.

\begin{example}
For the isentropic gas dynamics system \eqref{gasdyn}, $T=v$, $E=\sqrt{p'(u)}$ and
one choice of Riemann invariants is 
$
z = v - \sqrt{ p'(u)}/u$ and  $w = v  + \sqrt{ p'(u) }/u$.
For simplicity,
consider a power law pressure
with $p'(u) = u^{1/ \lambda} $.
Then 
$
z =  T -   E ^{1-2\lambda}$,  $w = T +   E ^{1-2\lambda}$. 
  Let $\beta = 1 - 2\lambda$, then we have 
$$
w_T = 1 = z_T, \quad w _E = \beta E^{\beta -1} = - z_E. 
$$
Now the functions defined in \eqref{earlier}
become
$$
B = 0 , \quad C =   - \beta E^{\beta -1 }  \circ \gamma  =   - \beta \zeta^{\beta -1 }  , \quad D = \frac{-1}{\beta} E^{1-\beta} \circ \gamma= \frac{-1}{\beta} \zeta ^{1-\beta} . 
$$
System \eqref{moregensys} becomes 
\begin{alignat}{2} 
 \xi_t&= - \beta \frac{\zeta^\beta }{\eta} \zeta_x \,,
  \quad &\xi(x,0)&= \xi_0(x) = T(U(x,0)) = v_0\,,\notag \\
 \zeta_t&=\frac{-\zeta^{2-\beta} }{\beta \eta}  \xi_x \,,
 \qquad &\zeta(x,0)&= \zeta_0(x) =E(U(x,0)) = u_0^{1/(1-\beta) } \,, \label{moregensys-j}\\
\eta_t&= \xi_x\,,\qquad\quad &\eta(x,0)&=  1\,.\notag
 \end{alignat}
We can define $h_1=\zeta_0^{1-\beta} =u_0$
and $   h_2 = \eta \xi$, and, after a calculation,
obtain a system of equations in conservation form:
 \begin{align} 
 \partial_ t    h_1 & = 0   \,, 
  \quad & h_1(x,0)&= u(x,0) \,, \notag  \\
\partial_ t  h_2&=\left( \frac12\frac{  h_2^2 }{\eta^2}  -   \frac{  \beta }{\beta+1}  \zeta( h_1, \eta) ^{ \beta+1}  \right) _ x   \,, 
  \quad & h_2(x,0)&= v(x,0) \,,\\ 
 \partial_ t\eta&= \left( \frac{  h_2 }{\eta}  \right) _x\,,\qquad\quad &\eta(x,0)&= 1\,.\notag
 \end{align}
 To obtain the weak diffeomorphism, set $\gamma$ to be the distributional solution to 
 $$
 \gamma_ t = \frac{  h_2}{\eta}  , \quad \gamma_x = \eta\,, 
 $$
 and define $u = \frac{  h_1  }{\eta}  \circ\gamma^{-1}  $ and $v = \frac{  h_2}{\eta}  \circ \gamma^{-1}$.  It is now easy to check that $u$ and $v$ are weak solutions to \eqref{gasdyn}.
 We  consider the weak formulation of  the equation for $u$: 
\begin{align*} 
&I_1 = \iint \psi _t u +\psi_x  u v  \,  dx\,dt \,. 
\end{align*}
Substituting the definitions  of  $u$ and $v$ gives
\begin{align*} 
&I_1 = \iint\psi _t  \frac{  h_1  }{\eta}  \circ\gamma^{-1}  +\psi_x    \frac{  h_1  }{\eta}  \circ\gamma^{-1}   \frac{  h_2}{\eta}  \circ \gamma^{-1}
 \,  dx\, dt .
\end{align*}
We apply the Radon-Nikod\'ym Theorem and 
define   $\widetilde \psi = \psi\circ \gamma$  so that as measures $\psi_t \circ \gamma = \widetilde 
\psi_t - \psi_x  \circ \gamma   \frac{  h_1}{\eta}   $. This leads to
\begin{align} 
&I_1 =  \iint \widetilde \psi _t     h_1     \,  dx dt =- \int \widetilde \psi (x,0)     u_0(x)    \,  dx .
\end{align} 
That $v$ is also a weak  solution can be verified similarly. 
Notice that the awkwardness of the formulation in this example is owing to the fact that \eqref{gasdyn} represents conservation of mass and velocity, rather than the conventional conservation of mass and momentum expressed in \eqref{continuity}.
\end{example}
}
 
\subsection{Correspondence Between Formulations}
The calculations in Section \ref{procedure} used a coordinate system in 
phase space based on Riemann invariants.
As the example just presented shows, the same pair of functions may be Riemann invariants for different, and inequivalent, systems of conservation laws. 
In this section we show that from a given pair of Riemann invariants, one can go forward, as we did in developing the particle path system \eqref{hsyst} in conservation form, or backward, to a conservation system in what we might call physical coordinates, in such a way that the two systems have equivalent weak solutions. 

We begin by recalling the well-known fact that any $2\times 2$ system in characteristic coordinates,
\begin{equation} \label{char}
\begin{pmatrix} z\\w\end{pmatrix}_t + \begin{pmatrix} \lambda(z,w)&0\\0&\mu(z,w)\end{pmatrix}
\begin{pmatrix} z\\w\end{pmatrix}_x = 
0\,,
\end{equation}
always corresponds to at least one system in conservation form, which is found by finding two
functionally independent solutions of the
linear hyperbolic equation
\begin{equation} \label{basic}
(\mu-\lambda)u_{zw}=\lambda_wu_z-\mu_zu_w\,.
\end{equation}
Specifically, if we seek solutions $u$ and $f$ of $u_t+f_x=0$
as functions of $(z,w)$, as in Proposition \ref{gen-case},
the calculation is
$$ u_zz_t+u_ww_t+f_zz_x+f_ww_x\equiv u_z(-\lambda z_x)+u_w(-\mu w_x)+f_zz_x+f_ww_x=0\,.$$
For this to hold for all functions $z$ and $w$ requires
$$ \lambda u_z =f_z\,,\quad\textrm{and}\quad \mu u_w=f_w\,,$$
and eliminating $f$ by equating the mixed partial derivatives gives
\begin{equation} \label{mixedps} (\lambda u_z)_w=(\mu u_w)_z\,,\end{equation}
which is just \eqref{basic} in conservation form,  the most useful form for the
next calculation.

We claim that \eqref{mixedps} is  the same equation as \eqref{hbareq} when the coordinate changes
are taken into account.
We begin with the Riemann invariant system,
\eqref{char} with $Z=(z,w)$ and $\Lambda = \textrm{diag\,}(\lambda,\mu)$.
Going from \eqref{basic} to \eqref{hbareq} is just a matter of tracking the coordinate changes.
A rotation  maps $(\lambda,\mu)$ to $(T,E)$,
while the assumption of functional independence gives the invertible mapping $(z,w)\mapsto(\lambda,\mu)$;
the composition of these two maps takes us from one system to the other.

The coefficients in \eqref{hbareq} are defined in terms of the partial derivatives $z_T$ and so on;
we have assumed that the Jacobian matrix $\partial(T,E)/\partial(z,w)$ is invertible, so
\begin{equation} \label{jay}
J=\begin{pmatrix}T_z&T_w\\E_z &E_w\end{pmatrix}=
\frac{1}{\Delta}\begin{pmatrix}w_E&-z_E\\-w_T & z_T\end{pmatrix}\,;\quad J^{-1}=
\begin{pmatrix} z_T&z_E\\w_T&w_E\end{pmatrix}\,,
\end{equation}
with $\Delta = z_Tw_E-z_Ew_T$ as before.
Since $\lambda = T-E$, $\mu=T+E$,
the defining equation for a conserved quantity in the physical system, \eqref{mixedps}, becomes
\begin{equation} \label{defin1}  \big((T-E) u_z\big)_w=\big((T+E) u_w\big)_z\,,\end{equation}
and since 
$$\partial_z=T_z\partial_T+E_z\partial_E=w_E/\Delta\partial_T-w_T/\Delta\partial_E\,,\quad 
\partial_w=T_w\partial_T+E_w\partial_E=-z_E/\Delta\partial_T+z_T/\Delta\partial_E\,,$$
\eqref{defin1}, in turn, becomes 
\begin{equation} \label{defin2} 
T_w\big((T-E) u_z\big)_T+E_w\big((T-E) u_z\big)_E=T_z\big((T+E) u_w\big)_T
+E_z\big((T+E) u_w\big)_E\,.
\end{equation}
Further calculations serve to express $u$ and its derivatives as functions of $E$ and $T$; we finally obtain
\begin{multline*}
-2Ez_Ew_E u_{TT} -2E(z_Ew_T+z_Tw_E) u_{TE} -2Ez_Tw_Tu_{EE}\\
+\bigg(-z_Tw_E+z_Ew_T-\Delta\bigg[(T-E)(\partial_T+\partial_E)(w_E/\Delta) 
 -(T+E)(\partial_T-\partial_E)(z_E/\Delta)   \bigg]\bigg)u_T\\
 +\bigg( z_Tw_E-z_Ew_T-2z_Tw_T-\Delta\bigg[(T-E)(\partial_T+\partial_E)(w_T/\Delta)
 +(T+E)(\partial_T+\partial_E)(z_T/\Delta)      \bigg]\bigg)u_E\\=0\,.
\end{multline*}
Comparing just the second-order derivatives, using the definitions \eqref{earlier} of $B$, $C$, and
$D$ (and dividing again by $\Delta$ and ignoring the composition with $\gamma$), we have
$$ -E(Cu_{TT}+2Bu_{TE}-Du_{EE})\,,$$
which coincides precisely with the characteristic form of the equation for $\bar h$, \eqref{charform}
(here we recall that $T\circ\gamma = \xi$ and $E\circ\gamma = \zeta$).
One can similarly show that the full equation agrees with \eqref{hbareq}.

Notice that we have not shown that every conserved quantity $h$ in the particle path system is
of the form $\eta \bar h$.
However,
we have a one-way implication: every conserved quantity for a system 
with Riemann invariant equation \eqref{char}, 
gives us a conserved quantity for the particle path system.

There are a number of examples where there is no obvious choice for a particle velocity, but where this procedure yields a weak diffeomorphism formulation. The equations of two-component chromatography are such a system, as are many of the Temple systems presented in Serre's monograph \cite{SerreII}.


\section{Conclusions}
We have shown that any scalar convex conservation law, and many systems of two conservation laws, can be given a particle path formulation that is valid for weak as well as for classical solutions.
The resulting particle paths take the form of ``weak diffeomorphisms'', or absolutely continuous mappings of the space variable with absolutely continuous inverses.
This correspondence extends to some systems of conservation laws a structure that has been exploited in some related equations.
Whether it will lead to further insights into conservation laws is an open question.

Another intriguing direction is the question of multi-dimensional conservation laws. The correspondence between Eulerian and Lagrangian formulations of fluid dynamics and elasticity in higher dimensions is also well-known \cite{CF,Dafbook}.
However, even for those basic equations the existence of this correspondence 
does not give immediate insight into the question of well-posedness, 
currently the major open problem in conservation laws.
 Further exploration will show whether deeper investigations into the correspondence 
 found for one-dimensional systems may lead to insights that will be relevant here.

\appendix
\section{A Note on the Existence of solutions}
\label{appendix}
In this paper, we have not dwelt on the well-posedness of the systems we study, because our emphasis has been elsewhere.
Following remarkable results of Bressan and his colleagues and students in the last few decades,
the theory of systems of conservation laws in a single space dimension is reasonably complete, and the function spaces in which well-posedness results can be formulated are well-understood, as summarized in the monographs of Bressan \cite{Bre:book} and of Dafermos \cite{Dafbook} and references therein. 

In this section, we give some background for well-posedness results on the systems we have examined. The theory is extensive, and this summary is necessarily incomplete.

Almost all the main results on systems in a single space dimension require that the system be strictly hyperbolic and that each characteristic speed be either genuinely nonlinear or linearly degenerate.
Genuine nonlinearity for systems is the generalization of convexity for a single conservation law, and our development in Section \ref{secfour} gives some hint about why this may be relevant to a particle path formulation.
Thus, for example,
the system \eqref{e-w-v} is well-defined and strictly hyperbolic as long as $v>0$
and $\eta>0$.
Let $S =\{(\eta,w,v)\mid \eta>0, v>0\}$.
The characteristic speeds are $\pm\sqrt{p'(v/\eta)}/\eta$ and $0$;
the first two are genuinely nonlinear under the usual assumptions about $p$
and the third linearly degenerate.
Under these hypotheses, Bressan's well-posedness theory gives global in time existence and
uniqueness of an admissible weak solution to the Cauchy problem on $\mathbb R$
with $(\eta,w,v)\in S$, provided $(\eta_0,w_0,v_0)$ is sufficiently close to
a constant in the total variation norm; the solution 
$(\eta,w, v) \in C([0, \infty); BV^3 )$.
For this particular problem, the hypotheses could most likely be weakened,
as the isentropic gas dynamics system has weak solutions without smallness
restrictions on the data.
Note that we want to avoid vacuum states, however.

The theory of weak solutions for periodic data has not been developed as completely, but we can cite
work of Young \cite{Yo1} and  Temple-Young \cite{TeYo1}, following Glimm-Lax \cite{GlLa}.
In addition, although it has been convenient to present \eqref{e-w-v} as a
system of three equations, it is equivalent to the system
\begin{equation}\label{e-w-0}
\begin{split}
\eta_t-\left(\frac{w}{\rho_0}\right)_x&=0\,,\\ 
w_t+\left(p\left(\frac{\rho_0}{\eta}\right)\right)_x&=0\,, 
\end{split}
\end{equation}
which does not fit into the classical conservation law theory, but is an example
of a system that has been studied by Dafermos and Hsiao \cite{DaHs} 
on the real line.

We note another difficulty that may arise in our examples:
 Both systems \eqref{sys2} and  \eqref{sys1} possess a zero eigenvalue with multiplicity two, violating the hypothesis of strict hyperbolicity.
 We can resolve this by
 considering, for example, an equivalent formulation of \eqref{sys1} as a system for $U=(\eta,w,s)^T$:  \begin{equation} \label{sys1:A}
\begin{split}
\eta_t-\left(\frac{w}{\rho_0}\right)_x&=0\,,\\ 
w_t+ (\alpha-1) \left( {\frac{s}{\eta} - \frac12 \frac{w^2}{\eta\rho_0  } }{  }\right)_x&=0\,, \\ 
s_t + (\alpha-1) \left( \frac{sw}{\eta \rho_0}   - \frac12 \frac{w^3}{\eta \rho_0^2}  \right)_x & = 0\,.
\end{split}
\end{equation}
Write this as
\begin{equation}  
U_ t + \frac{\partial f (U,x)}{\partial x }  = 0.
\end{equation}

One can show that the Jacobian matrix, $G =  f_U(U,x)$, has three distinct eigenvalues uniformly separated as well as bounded away from zero. 
The result found in  Dafermos and Hsiao \cite{DaHs} can now be applied.
To verify that their result is applicable to \eqref{sys1:A}, we verify conditions (1.4)-(1.9) found in their manuscript. 
The first condition, condition (1.4) is trivially satisfied: $ \frac{\partial f (0,x)}{\partial x }  = 0$. Condition (1.5) is the condition that the system possesses three distinct eigenvalues, which we have already verified; while condition (1.6) is that they are bounded away from zero which can be accomplished by a change of variables. 
Condition (1.7) is that there exists a constant $a \in \mathbb R$ such that
$$
|G| < a,\quad \text{ and }\quad | f_{UU}| < a,
$$ 
for all $(\eta, w, s) \in \mathcal B$; which is easily accomplished. Finally, we verify conditions (1.8) and (1.9) by calculating 
$$
f_{Ux}   =   \begin{bmatrix}
0 & \frac{\rho_0'}{(\alpha-1)\rho^2_{0}}   & 0\\
\frac{-w^2\rho_0' }{2\eta^2 \rho^2_{0}}  - \frac{s }{\eta^2} & \frac{   w \rho_0'  }{\eta \rho^2_{0}}  &   \frac{1}{\eta} \\ 
\frac{sw\rho_0'}{\eta^2\rho^2_{0}} - \frac{w^3\rho_0'}{\eta^2\rho_{0}^3} \ & \  -\frac{s\rho_{0}'}{\eta \rho_{0}^2} +\frac{3w^2\rho_{0}'}{ \eta \rho_{0}^3}   \ & \ -\frac{w\rho_{0}'}{\eta \rho^2_{0}} 
\end{bmatrix} \,.
$$
Assuming that the compactly supported initial data $\rho_0$ are smooth  and bounded away from zero, it is easy to see that 
$
|f_{Ux}| < K(x)
$
for a smooth (bounded) integrable function $K(x)$ which depends on $\rho_0$. 
The application of the result in \cite{DaHs} requires $\rho_0$ to be smooth, this assumption could possibly be weakened. However, as the aim of this paper is not to provide a new result on the existence of solutions to conservation laws, we will not explore whether or not Glimm's scheme can be applied directly to our system assuming $\rho_0$ is only BV. Thus, the hypotheses of Theorem 1 in Dafermos and Hsiao \cite{DaHs} have been met, and the question of existence and uniqueness of global in time weak solutions to the Cauchy problem  \eqref{sys2} has been answered in the affirmative. 
\section{Proof of Proposition \ref{entropy-condition}}
\label{B}
A standard argument gives the form of the admissibility pairs, and to determine if $\mathcal E$ is convex, we need only to determine if $\rho X(p\rho^{-\alpha})$ is convex in the variables $\rho$, $m = \rho u$, $\epsilon = \rho E$. In these variables, every admissibility  pair is given by
\begin{equation}
\mathcal E = \rho X\left(\frac{\epsilon}{\rho^\alpha} - \frac{m^2}{2\rho^{\alpha+1} }\right)\,, \quad \mathcal Q = m X\left(\frac{\epsilon}{\rho^\alpha} - \frac{m^2}{2\rho^{\alpha+1} }\right)\,.
\end{equation}
Let $z (\rho, m, \epsilon)= \frac{\epsilon}{\rho^\alpha} - \frac{m^2}{2\rho^{\alpha+1} }$;   $\mathcal E $ is convex if its Hessian, $H(\mathcal E)$, is positive definite. This is equivalent to  choosing $X'<0$, $X''>0$ and showing that  the matrix  
\begin{align} M:=
   \begin{bmatrix}
\rho(z_{\rho \rho} - \frac{ 2 z_\rho z_{\rho \epsilon} }{z_\epsilon})& \rho (z_{\rho m} - \frac{z_m z_{\rho \epsilon}}{z_\epsilon})  & z_\epsilon + \rho z_{\rho \epsilon}\\
\rho(z_{\rho m} - \frac{   z_m z_{\rho \epsilon} }{z_\epsilon})& \rho z_{mm}&   0\\ 
z_\epsilon + \rho z_{\rho \epsilon} \ & \  0  \ & \ \rho z_\epsilon^2 \frac{X''}{X'} 
\end{bmatrix}\,,
\end{align}
is negative definite. We first compute the partial derivatives 
\begin{align*}
z_\rho  & = -\frac{\alpha \epsilon }{\rho^{\alpha+1}} + \frac{(\alpha+1) m^2}{2 \rho^{\alpha+2} } \hskip1in  z_m = -\frac{m}{\rho^{\alpha+1}}
 \hskip1in  
z_\epsilon = \frac{1}{\rho^{\alpha}}
\\
z_{\rho \rho } &= \frac{\alpha(\alpha+1) \epsilon}{\rho^{\alpha+2}} - \frac{(\alpha+1)(\alpha+2) m^2}{2 \rho ^{\alpha+3} } 
 \hskip0.32in   z_{\rho m} = \frac{(\alpha+1)m}{\rho ^{\alpha+2}}  \hskip0.78in   z_{\rho \epsilon} = -\frac{\alpha }{\rho ^{\alpha+1}} 
\end{align*}
Therefore, 
\begin{align*}
 z_{\rho \rho} - \frac{ 2 z_\rho z_{\rho \epsilon} }{z_\epsilon} & = 
\frac{\alpha(\alpha+1) \epsilon}{\rho^{\alpha+2}} - \frac{(\alpha+1)(\alpha+2) m^2}{2 \rho ^{\alpha+3} } 
+2 \frac{\alpha}{\rho}\left( -\frac{\alpha \epsilon }{\rho^{\alpha+1}} + \frac{(\alpha+1) m^2}{2 \rho^{\alpha+2} } \right)
\\ & =
\frac{\alpha(1-\alpha) \epsilon}{\rho^{\alpha+2}} + \frac{(\alpha+1)(\alpha-2) m^2}{2 \rho ^{\alpha+3} }  ,
\end{align*}
which is negative for all $1<\alpha < 2$,   is a natural assumption for physical systems. Next we verify 
$$
\rho^2 z_{mm}(z_{\rho \rho} - \frac{ 2 z_\rho z_{\rho \epsilon} }{z_\epsilon}) -\rho^2 (z_{\rho m} - \frac{z_m z_{\rho \epsilon}}{z_\epsilon}) ^2 > 0. 
$$
Since $\rho^2>0$, we may ignore this term and we calculate 
\begin{align*}
 z_{mm}(z_{\rho \rho} - \frac{ 2 z_\rho z_{\rho \epsilon} }{z_\epsilon}) 
& = 
\frac{\alpha(\alpha-1) \epsilon}{\rho^{2\alpha+3}} + \frac{(\alpha+1)(2-\alpha) m^2}{2 \rho ^{2\alpha+4} } 
\\
( z_{\rho m} - \frac{z_m z_{\rho \epsilon}}{z_\epsilon} )^2
& = 
\left( \frac{(\alpha+1)m}{\rho ^{\alpha+2}}   - \frac{m \alpha }{\rho ^{\alpha+2} } \right)^2 
 = 
\frac{\alpha^2 m^2   }{\rho ^{2\alpha+4} } + \frac{(\alpha+1)^2m^2}{\rho ^{2\alpha+4}}  -2 \frac{\alpha(\alpha+1)m^2}{\rho ^ {2\alpha+4}}   =  \frac{ m^2}{\rho ^ {2\alpha+4}}  \,,
\end{align*}
and therefore, 
\begin{align*}
 z_{mm}(z_{\rho \rho} - \frac{ 2 z_\rho z_{\rho \epsilon} }{z_\epsilon}) - (z_{\rho m} - \frac{z_m z_{\rho \epsilon}}{z_\epsilon}) ^2
& = 
\frac{\alpha(\alpha-1) \epsilon}{\rho^{2\alpha+3}} + \frac{(\alpha+1)(2-\alpha) m^2}{2 \rho ^{2\alpha+4} }  -  \frac{ m^2}{\rho ^ {2\alpha+4}}
\\ & = 
\frac{\alpha(\alpha-1)  }{2\rho^{2\alpha+4}}  (2\rho \epsilon - \  m^2) 
 \,,
\end{align*} 
which is positive for all physical $\rho, m$ and $\epsilon$. 

The determinant of the entire matrix is 
$$
\det M = (z_\epsilon + \rho z_{\rho \epsilon} )^2 \rho z_{mm} + \rho z_{e}^2  \frac{X''}{X'}  [ \rho^2 z_{mm}(z_{\rho \rho} - \frac{ 2 z_\rho z_{\rho \epsilon} }{z_\epsilon}) -\rho^2 (z_{\rho m} - \frac{z_m z_{\rho \epsilon}}{z_\epsilon}) ^2 ]\,.
$$
Thus, we find 
\begin{align*}
\det M  = -\frac{(1-\alpha)^2}{\rho^{3\alpha} }  +
  \frac{\alpha(\alpha-1) X'' }{2\rho^{2\alpha+1}X'}  (2\rho \epsilon - \  m^2) \,,
\end{align*}
which is negative for all physical values of $\rho, m$ and $\epsilon$ and any $X$ such that $X'<0<X''$. 

\subsection*{Acknowledgement}
We are grateful to a referee who suggested that the terminology ``convex extension'', coined by Friedrichs and
Lax in \cite{FrLa} replace an awkward expression we had been using.

\end{document}